\documentclass[letterpaper,11pt]{amsart}

\usepackage{color,amscd}
\usepackage[draft]{changes}
\usepackage{enumerate}
\usepackage{mathtools}
\usepackage[abbrev]{amsrefs}
\usepackage{amssymb,amsmath,amsthm,xfrac}

\newcommand{\pa}{\partial}

\makeatletter
\@namedef{Changes@AuthorColor}{red}
\colorlet{Changes@Color}{red}
\makeatother

\newtheorem{thm}{Theorem}[section]
\newtheorem{prop}[thm]{Proposition}

\newtheorem{Def}[thm]{Definition}
\newtheorem{lem}[thm]{Lemma}

\newtheorem{remark}[thm]{Remark}
\newtheorem{corl}[thm]{Corollary}
\newtheorem{conjecture}[thm]{Conjecture}
\newtheorem{example}[thm]{Example}

\newcommand{\Nat}{\mathbb{N}}

\newcommand{\eps}{\varepsilon}
\newcommand{\Hilbert}{\mathcal{H}}
\newcommand{\Dom}{\mathfrak{Dom}}

\def \<{\langle}
\def \>{\rangle}

\def \R{\mathbb R}

\def \H{{\cal H}}

\def \H^0{{\cal H}^0 or}

\def \w{\omega}

\def \n{\nabla}
\def \beq{\begin{equation}}
\def \eeq{\end{equation}}

\def \n{\nabla}

\def \eref{\eqref}

\begin{document}



\title[Form Spectrum in Vanishing Curvature]{connected essential spectrum: \\the case of differential forms}

\author{Nelia Charalambous}
\address{Department of Mathematics and Statistics, University of Cyprus, Nicosia, 1678, Cyprus} \email[Nelia Charalambous]{nelia@ucy.ac.cy}

\author{Zhiqin Lu} \address{Department of
Mathematics, University of California,
Irvine, Irvine, CA 92697, USA} \email[Zhiqin Lu]{zlu@uci.edu}

\thanks{
The first author was partially supported by a University of Cyprus  Internal grant. The second author is partially supported by the DMS-19-08513.}
 \date{\today}



\begin{abstract}
In this article we prove  that, over complete manifolds of dimension $n$ with vanishing curvature at infinity, the  essential spectrum of the Hodge Laplacian on differential $k$-forms is a connected interval  for $0\leq k\leq n$.  The main idea is to show that large balls of these manifolds, which capture their spectrum, are close in the Gromov-Hausdorff sense to product manifolds. We achieve this by carefully describing the collapsed limits of these balls. Then, via a new generalized version of the classical Weyl criterion, we demonstrate that very rough test forms that we get from the $\eps$-approximation maps can be used to show that the essential spectrum is a connected interval.  We also prove that, under a weaker condition  where  the Ricci curvature is asymptotically nonnegative,  the essential spectrum on $k$-forms is $[0,\infty)$, but only for $0\leq k\leq q$ and $n- q\leq k\leq n$ for some integer $ q\geq 1$ which depends the structure of the manifolds at infinity.
\end{abstract}

\maketitle

\section{Introduction}

In this article, we study the spectrum  of the Hodge Laplacian on differential forms over a complete orientable Riemannian manifold. The spectrum, $\sigma(H)$, of a nonnegative  and self-adjoint operator $H$    consists of all points $\lambda\in \mathbb{C}$ for which $H-\lambda I$ fails to be invertible.  $\sigma(H)$ is in fact a subset of the nonnegative real line. The essential spectrum, $\sigma_\mathrm{ess}(H)$, consists of the cluster points in the spectrum and of isolated eigenvalues of  infinite multiplicity.  The discrete isolated  spectrum, $\sigma_{\rm dis}(H)$,  defined by $\sigma(H)\backslash\sigma_\mathrm{ess}(H)$, consists of isolated eigenvalues of finite multiplicity.

It is well-known that the Hodge Laplacian $\Delta$ on $k$-forms over a Riemannian manifold is a densely defined,  nonnegative  and self-adjoint operator.
If the manifold is compact, then $\sigma_\mathrm{ess}(\Delta)=\emptyset$, that is, $\sigma(\Delta)=\sigma_{\rm dis}(\Delta)$ . On the other hand, when the manifold is  noncompact  both types of spectra, $\sigma_{\rm dis}(\Delta)$ and $\sigma_{\rm ess}(\Delta)$,  may exist. We call the essential spectrum  \emph{computable} when  it  is a connected set, that is, when it  is the empty set or  an interval $[\alpha,\infty)$ for some $\alpha\geq 0$.

The main goal of this article is to prove that the essential spectrum of the Hodge Laplacian over a complete non-compact Riemannian manifold with vanishing curvature  is computable.  The computation of the spectrum (or essential spectrum) of the Laplacian requires the construction of a large class of test differential forms (cf.~\cites{Ant,CharJFA,EF93}). This is a difficult task on a general manifold, since there hardly exist any canonically defined differential forms to work with.    In order to overcome this difficulty, we make systematical  use of  collapsing theory in Riemannian geometry to  construct a large class of such forms. These test forms are not approximate eigenforms, but they satisfy a new generalized version of the Weyl criterion.
 Our methodical use of collapsing theory for locating the essential spectrum is the most technically subtle part of our paper, and lies at the core of the computation of the essential spectrum on forms.

\begin{Def} \label{DefAF}
A manifold is called asymptotically flat, or with vanishing curvature, if its curvature tensor tends to zero at infinity.
\end{Def}
Note that our definition of asymptotic flatness is different from the one used in General Relativity. We emphasize that we make no assumptions on the  curvature  decay rate,  nor do we make any assumptions  on the volume growth  (or decay) of the manifold.   The absence of strong curvature decay assumptions, as in our case, is a barrier to obtaining smooth coordinates at infinity that could be used when computing the spectrum (cf. \cites{BKN,LoSh}).

Here  and for the rest of the paper, we assume that $M$ (or $M^n$)  is a complete noncompact Riemannian manifold of dimension $n$.  For $0\leq k\leq n$, we shall use either $\sigma(k,\Delta,M)$, or $\sigma(k,\Delta)$ ($\sigma_{\rm ess}(k,\Delta,M)$, or  $\sigma_{\rm ess}(k,\Delta)$ \emph{resp}.) to denote the spectrum (essential spectrum \emph{resp}.) of the Laplacian on $k$-forms over the  manifold $M$.  The first  main result of this paper is the following.
\begin{thm}\label{thm1}
Let $(M^n,g)$ be an asymptotically flat complete noncompact Riemannian manifold.  Then for $0\leq k\leq n$,  $\sigma_\mathrm{ess}(k,\Delta, M)$  is either empty or and interval $[\alpha_k,\infty)$ for a nonnegative number $\alpha_k$, in other words it is computable.
\end{thm}

The behavior of the spectrum becomes more complicated when we  only assume that the manifold has \emph{asymptotically nonnegative Ricci curvature},  that is, when we assume  that
\[
\liminf_{d(x,p)\to\infty}{\rm Ric}(x)\geq 0
\]
for some fixed point $p\in M$.  In Section \ref{10} we will study the $k$-form spectrum over noncompact manifolds under a weaker assumption: that they have asymptotically nonnegative Ricci curvature on a sequence of expanding balls.
\begin{Def}\label{76} Let $M$ be a complete noncompact Riemannian manifold. We say that $M$ has asymptotically nonnegative  Ricci curvature along a sequence of expanding balls,  if it contains a sequence of disjoint balls $M_i=B_{x_i}(R_i)$  with $ x_i\to \infty$ and  $R_i \to \infty$  such that
\[
\mathrm{Ric}_{M_i} \geq - \delta_i
\]
where $\delta_i\to 0$ with $\delta_i\geq 0$.
\end{Def}

We will show
\begin{thm} \label{thmRic}
Let $(M^n,g)$  be a  complete noncompact  Riemannian manifold with asymptotically nonnegative Ricci curvature   along a sequence of expanding balls.
Moreover, assume that  the dimension of the manifold at infinity is $q$ (see Definition~\ref{77}). Then for each $k\leq q$ and $k\geq n-q$
\[
\sigma(k,\Delta, M)=\sigma_{\rm ess}(k,\Delta, M) = [0,\infty).
\]
In particular, if $q\geq n/2$, then for all $k$ the spectrum is $[0,\infty)$.
\end{thm}

Observe that Definition \ref{76} is satisfied when the manifold has nonnegative Ricci curvature. In Section \ref{10} we also prove the following result.

\begin{corl}\label{cor14}
Let $(M^n,g)$  be a noncompact  complete  Riemannian manifold with nonnegative Ricci curvature.
Assume that there is a constant number $c>0$ and an integer $s\geq 1$ such that at a fixed point $p$
\begin{equation}\label{c5}
{\rm Vol}\,(B_{p}(R))\geq cR^s.
\end{equation}
Then the dimension of $M$ at infinity is at least $s$. Consequently, for each $k\leq s$  and $k\geq n-s$
\[
\sigma(k,\Delta, M)=\sigma_{\rm ess}(k,\Delta, M) = [0,\infty).
\]
\end{corl}

Under the assumption of Theorem \ref{thmRic}, the function spectrum  is $[0,\infty)$. However, this is not the case for the $k$-form essential spectrum.  The bottom of the essential spectrum for $k$-forms may capture further details of  the geometric and topological structures of the manifold  that are ignored by the function spectrum \cites{Don81,Don}. This fact will be reflected in  the proof of the main results of this paper, and in the examples of Section \ref{11}.

Based on the above results, we would make the following conjecture.
\begin{conjecture}\label{16}
Let $(M^n,g)$  be a  complete noncompact  Riemannian manifold with asymptotically nonnegative Ricci curvature. Then its  $k$-form essential spectrum   is computable, that is, for each $0\leq k\leq n$  either $\sigma_{\rm ess}(k,\Delta, M) =\emptyset$ or $\sigma_{\rm ess}(k,\Delta, M) = [a_{k},\infty)$ for some  $a_{k}\geq 0$.  \end{conjecture}
Since so far very little is known about the structure of manifolds with nonnegative Ricci curvature in the collapsing case, we believe that it is not possible to use
our the method to prove the conjecture. On the other hand, the conjecture is closely related to many papers on the spectral gap (see for example \cites{CLL,Lott2,Post1,Post2,SchoTr} for the function case, and  ~\cite{LL-2} for the differential form case). So far the presence of gaps in the essential spectrum has only been demonstrated in cases of manifolds  that have curvature bounded above by a negative constant on an unbounded sequence of points.

Manifolds with computable essential spectrum are in a sense ``opposite'' to those whose essential spectrum has gaps, that is, whose essential spectrum is a disconnected subset of $\R^+$.
In \S~\ref{11}, in addition to providing some examples, we make the important observation that  could lead to the proof of Conjecture \ref{16}.  We will further address the observation in \S~\ref{11} in forthcoming work.

Our results are  related to the spectral continuity  results derived from the Cheeger-Fukaya-Gromov  and Cheeger-Colding theories.  All of the above references however have only considered the compact case (see for example Fukaya~\cite{Fuk2} and Cheeger-Colding~\cite{CCoIII} for the function spectrum and Lott~\cites{lott, lott1}, and Honda~\cite{Ho}, on the form spectrum).   Our  results are also related to the work of Dodziuk who considered the behavior of the spectrum under continuous deformations of the metric in the compact case \cite{dod}, but also to our article \cite{ChLu6} which treats the noncompact case.

Noncompact manifolds are more complicated both topologically and geometrically when compared to  compact ones. However, with respect to the continuity of the \emph{essential spectrum} we have an additional technique which  is not available in the compact case: we can zoom in over a large portion of an asymptotically flat (or asymptotically Ricci nonnegative) manifold and still be able to control its curvature. This unique feature allows us to study the continuity of the essential spectrum  in both the collapsing and non-collapsing cases.

Throughout the paper, we use the following notation.  We denote by $(M,p,g)$ the pointed manifold $M$ centered at $p$ and endowed with the metric $g$. The open geodesic ball of radius $R$ centered at $p$ will be denoted by $B_p(R)$. In order to keep track of the radius and metric of a geodesic ball after rescaling, we will denote by $(M, p, g, R)$  the manifold $(B_p(R),p,g)$.  When it is otherwise clear we could denote $(M,p,g)$  by $(M,p)$, or $M$.

We will also use the notation of Cheeger and Colding \cite{CCoII} to denote by
\begin{equation}\label{psi}
\Psi(u_1,\cdots, u_r \mid  C_1,\cdots, C_s)
\end{equation}
 any positive  function that depends on $u_1,\cdots, u_r$ and additional parameters $C_1,\cdots, C_s$, such that when the additional parameters remain fixed then
\[
\lim_{u_1,\cdots, u_r\to 0} \Psi(u_1,\ldots, u_r\mid  C_1,\cdots, C_s) =0.
\]
This notation proves to be convenient. For example, under the above notation a sequence of real numbers $a_i\to a$ as $i\to\infty$ can be written as $|a_i-a|=\Psi(i^{-1})$.

The organization of this paper is as follows. In Section~\ref{sec2},  we show that on manifolds with asymptotically nonnegative Ricci curvature the bottom of the essential spectrum is captured by a sequence of expanding geodesic balls with  given radii. In Section \ref{7} we use the generalized Weyl criterion to show the continuity of the spectrum under a very general notion for the $\mathcal{C}^0$ topology of  the metrics.
 The technical heart of this paper is in Sections~\ref{S4} and ~\ref{S5}.  In  Section~\ref{S4}, we prove Theorem~\ref{thm1} under the additional assumption that there is a minimal sequence collapsing to a smooth manifold, whereas in Section~\ref{S5}, we prove that such type of ``good'' minimal sequence always exists. The  complete  proof of Theorem \ref{thm1}  can be  found in Section \ref{S6}, and we prove Theorem \ref{thmRic} and its Corollary in Section \ref{10}.
The generalized Weyl criterion (Theorem \ref{Thm.Weyl.bis-4}) we prove is  more  powerful than its preliminary version in~\cite{CLL} and is the one we will apply throughout this paper. We provide its proof in the Appendix. \\

{\bf Acknowledgement.} The authors would like to thank  Kenji Fukaya, John Lott,  Rafe Mazzeo, and Xiaochun Rong  for their useful comments and suggestions.

\section{Localization of the bottom of the essential spectrum}\label{sec2}

Let $(M^n,g)$ be a complete $n$-dimensional Riemannian manifold. The metric $g$ induces a pointwise inner-product $\< \cdot ,\cdot \>$ on the space of $k$-forms $\Lambda^k(M)$. We denote the $L^2$  inner product  as $  (\cdot , \cdot )=\int_M \< \cdot  ,\cdot \>$   and the  corresponding $L^2$ norm as $\|\cdot\|_{L^2}$. Let $L^2(\Lambda^k(M))$  be  the space of $L^2$ integrable $k$-forms and  $\Delta_k$ be  the Laplacian on $k$-forms  as well as its Friedrichs extension on $L^2(\Lambda^k(M))$.  Then $\Delta_k$
is a  densely defined,  self-adjoint and nonnegative operator on the Hilbert space $L^2(\Lambda^k(M))$. We shall write $\Delta$ instead of $\Delta_k$ whenever it is clear.

\begin{lem} \label{lem7b}
Let $M$ be a complete Riemannian manifold.  We
assume that the sequence of balls   $B_{x_i}(R_i)$  satisfies  the following properties

\begin{enumerate}
\item
$x_i\to \infty$ and  $R_i   \to \infty  \ \ \text{as}  \ \ i\to \infty$,
\item For any compact subset $K$ of $M$, there is an $i_o(K)$ such that if $i>i_o(K)$, then $B_{x_i}(R_i)$ is outside that compact set.

\item
$\mathrm{Ric} \geq   - \delta_i  \  \text{on} \,\,B_{x_i}(2R_i)$ for a sequence of $\delta_i\to 0$.
\end{enumerate}
Then for any sequence $R_i'<R_i$ with $R_i'\to\infty$, there exists a sequence $y_i\in B_{x_i}(R_i)$ such that
\begin{equation}\label{890}
\liminf_{i\to \infty} \lambda_o (k, B_{y_i}(R_i')) \leq  \liminf_{i\to \infty} \lambda_o (k, B_{x_i}(R_i)),
\end{equation}
where $\lambda_o (k, U)$ denotes the smallest eigenvalue of the Friedrichs  Laplacian on $k$-forms over a bounded open set $U\subset M$.
\end{lem}

\begin{proof}  By the monotonicity of the Friedrichs eigenvalue  with respect to the domain,
we may assume that $R_i'<\min(R_i,\delta^{-1/2}_i)$ without loss of generality.

We consider the \emph{maximum number} of points $y_1,\cdots, y_{N_i}\in B_{x_i}(R_i)$ such that
\begin{enumerate}
\item $B_{y_j}(R_i')$ are disjoint;
\item the collection of open balls $\{B_{y_j}(2 R_i')\}$ is an open cover of $B_{x_i}(R_i)$.
\end{enumerate}

Such a covering is called a Gromov covering. One of the features of the Gromov covering is that, by assumption {\it (3)}, there exists an absolute constant  $C(n)$ with the property that  any $p\in B_{x_i}(R_i)$ is covered by at most $C(n)$ open balls from the collection $\{B_{y_j}(2R_i')\}$.

 Let $\{\rho_j^2\}$ be  the partition of unity subordinate to the covering of $B_{x_i}(R_i)$ such that $\sum_j \rho_j^2 =1$ and $|\nabla \rho_j|^2\leq C/(R_i')^2$.
 For any $\eps>0$, let $\omega=\omega_i$ be
 a smooth $k$-form  with compact support in $B_{x_i}(R_i)$
such that
\[
 (\Delta \w, \w)\leq (\lambda_o(k, B_{x_i}(R_i))+\eps)\|\w\|^2.
\]
Setting $\w_j =\rho_j \w $, and given that  $\sum_j \rho_j \n \rho_j = 0$ in $B_{x_i}(R_i)$, we have
\begin{equation}\label{gromov}
\sum^{N_i}_{j=1}(\Delta \w_j, \w_j)  \leq  C\sum^{N_i}_{j=1}|\nabla\rho_{j}|^2 \, \| \w \|^2   +  (\Delta \w, \w)\leq  \left(\frac{C}{(R_i')^2} \,+\lambda_o(k, B_{x_i}(R_i))+\eps\right) \| \w \|^2.
\end{equation}
Then there exists one $j =j(i)$ such that
\[
(\Delta \w_j, \w_j)  \leq \left({C}{(R_i')^{-2}} \,+\lambda_o(k, B_{x_i}(R_i)\,)+\eps \right) \,\| \w_j \|^2.
\]
Thus
\[
\lambda_o (k, B_{y_j}(R_i')) \leq {C}{(R_i')^{-2}} \,+\lambda_o(k, B_{x_i}(R_i))+\eps.
\]
Renaming $j=j(i)$ by $i$, we prove the lemma by letting $\eps\to 0$ and $R_i'\to\infty$.
\end{proof}

\begin{Def} \label{defMin}
Let $S=\{(x_i, R_i)\}$ be a sequence, where $x_i\in M$ and $R_i>0$ are real numbers.  We say that $S$ is a minimal sequence (for   $k$-forms), if
\begin{enumerate}
\item $R_i\to\infty$;
\item the balls $B_{x_i}(R_i)$ are disjoint;
\item for any compact subset $K$ of $M$, there is a number $i=i(K)$ such that if $i>i(K)$, then $B_{x_i}(R_i)\cap K=\emptyset$;
\item we have
\[
\lim_{i\to\infty} \lambda_o(k, B_{x_i}(R_i))=\lambda_o^{\rm ess}(k, M),
\]
\end{enumerate}
where $\lambda_o^{\rm ess}(k, M)$ is the bottom of the essential spectrum of the Laplacian on $k$-forms.\footnote{If the essential spectrum is an empty set, we define $\lambda_o^{\rm ess}(k, M)=\infty$.}
\end{Def}

 \begin{prop}\label{prop26}
Let $M$ be a complete non-compact Riemannian manifold with asymptotically nonnegative Ricci curvature. Let $R_i$ be a sequence of real numbers such that $R_i\to\infty$. Then there exists a sequence $\{y_i\}_{i=1}^\infty\in M$ such that
$S=\{(y_i, R_i)\}$ is a minimal sequence, and moreover, the balls $B_{x_i}(2R_i)$ as in Lemma \ref{lem7b} are disjoint.
\end{prop}

\begin{proof} Let $\tilde S=\{(x_i, \tilde R_i)\}$ be a minimal sequence. Since the balls $B_{x_i}(\tilde R_i)$ leave any compact subset of $M$ for $i$ large enough, we can assume that   $\mathrm{Ric} \geq   - \delta_i$  on $B_{x_i}(\tilde R_i)$ for a sequence of $\delta_i\to 0$. For any sequence $R_i\to\infty$, by passing to a subsequence if necessary, we may assume that $R_i<\tilde R_i$. Moreover, since $R_i$ is given, using assumption {\it (3)} for a minimal sequence we can inductively choose a sequence $\{x_i\}$ that satisfies
\begin{equation}\label{745}
d(x_i, x_j)>4R_i+\tilde R_j+\tilde R_i,
\end{equation}
for any $j<i$.

By Lemma~\ref{lem7b}, there is  a sequence $S=\{(y_i, R_i)\}$ such that ~\eqref{890} is valid, where $y_i\in B_{x_i}(R_i)$.
Inequality~\eqref{745} implies a stronger type of separation, namely, that  the geodesic balls  $B_{y_i}(2R_i)$ are disjoint.
 Since the $B_{y_i}(R_i)$ are disjoint, then $S$ must be a minimal sequence. This completes the proof of the proposition.
\end{proof}

Minimal sequences always exists. However, the size of the expanding balls is not easy to control. Therefore,
the following is an interesting conjecture.
\begin{conjecture}
Let $M$ be a complete Riemannian manifold with Ricci curvature bounded from below.
Given a sequence of real numbers $R_i\to\infty$, we can  find a sequence $y_i\in M$ such that $\{(y_i, R_i)\}$ is a minimal sequence.\end{conjecture}

By Proposition \ref{prop26}, if $M$ has asymptotical nonnegative Ricci  curvature, then the conjecture is true.  The conjecture is also true when $M$ is a hyperbolic space form.

\section{Continuity of the Spectrum in the Noncollapsing Case}\label{7}

 In this section we include some known results about the form spectrum of product manifolds and flat manifolds. Then we will see how to obtain the computability of the spectrum when the minimal sequence is close to a product space.

\begin{Def} \label{lem81} Let $m<n$ be a positive  integer. Let $K^{n-{m}}$  be a compact flat manifold of dimension $n-m$. As before, denote by ${\lambda_o(l,K^{n-m})}$
the smallest eigenvalue of the Laplacian $\Delta$ on $l$-forms over $K$ for $0\leq l \leq n-m$. By  Poincar\'e duality, we have  $\lambda_o(l,K^{n-m})=\lambda_o(n-m-l,K^{n-m})$.

Define $\alpha( K,m,n,k)=0$ when $m\geq n/2$, or when $m<n/2$  and  either $0\leq k\leq m$, or $n-m\leq k\leq n$. When $m+1\leq k\leq n/2$ define $\alpha( K,m,n,k)=\min\{ \,\lambda_o(k-l,K^{n-m}) \mid 0\leq l\leq m\}$,  and set $\alpha(K,m,n,k)=\alpha(K,m,n,n-k)$, when $n/2< k\leq n-(m+1)$.
\end{Def}

For $0\leq k\leq n$, a simple characterization of $\alpha( K,m,n,k)$ is given by the following  (cf.~\cite{ChLu5})
\begin{equation}\label{657}
\sigma(k,\Delta, K^{n-m}\times \mathbb{R}^{m})=\sigma_{\rm ess}(k,\Delta, K^{n-m}\times \mathbb{R}^{m})=[\alpha(K,m,n,k),\infty).
\end{equation}
 In particular, $\alpha(K,m,n,k)$ is the bottom of the spectrum over $K^{n-m}\times \mathbb{R}^{m}$ for $k$-forms.

\begin{thm}[\cite{ChLu5}] \label{thmF2}
Let $X=\mathbb{R}^n/\Gamma$ be a flat noncompact Riemannian manifold, and  $0\leq k\leq n$. Then there exist a nonnegative integer $m$ and  a compact flat manifold $K^{n-m}$ such that
\[
\sigma(k,\Delta, M)=\sigma_{\rm ess}(k,\Delta, M)=\sigma(k,\Delta, K^{n-m}\times \mathbb{R}^{m})=[\alpha(K,m,n,k),\infty).
\]
\end{thm}

The following result allows to compute the essential spectrum of a noncompact manifold whenever the minimal sequence is close, in the Gromov-Hausdorff sense, to a sequence of product manifolds of the type compact manifold times a Euclidean space. We shall apply this result in the case that the limit of the minimal sequence is a smooth manifold. The proof adapts an argument of Dodziuk in \cite{dod} to the noncompact case, and uses  similar arguments as our paper~\cite{ChLu6}  as well as  the generalized Weyl criterion, Theorem~\ref{Thm.Weyl.bis-4}.  For the sake of completion we include the proof in the Appendix.

\begin{thm} \label{thmSpecGH}
Let $(M,g_M)$  be  a  complete noncompact manifold, and  $M_i=B_{x_i}(R_i)$ a sequence of disjoint geodesic balls in $M$ with $R_i\to\infty$. Let $N_i$ be a sequence of compact Riemannian manifolds whose diameter is uniformly bounded  and denote by $g_i$ the product metric on $N_i\times \R^{m}$.  Assume  that there exist smooth maps
\[
f_i : B_{x_i}(R_i) \to N_i\times \R^{m}
\]
which are differomorphims onto their image, and which
 satisfy
\begin{equation} \label{thm7e1}
|g_{M_i}-f_i^*g_i| \leq \Psi(i^{-1}).
\end{equation}

Then for any $0\leq k\leq n$
\[
\sigma_{\rm ess}(k,\Delta, M) \supset  [\alpha_k,\infty),
\]
where
\[
\alpha_k=\liminf_{i\to\infty} \alpha(N_i,m,n,k),
\]
and $\alpha(N_i,m,n,k)$ is as in Definition~\ref{lem81}.  If in addition, $S=\{(x_i, R_i)\}$ is a minimal sequence, then
\begin{equation}\label{279}
\sigma_{\rm ess}(k,\Delta, M) =[\alpha_k,\infty).
\end{equation}
In particular, in this case, we have $\alpha_k=\lambda_o^{\rm ess}$.
\end{thm}

Note that we \emph{do not require} the convergence of the sequence $N_i$. This is important when proving the results in the next two sections.

\section{The collapsing case: smooth limit} \label{S4}

In this section, we prove Theorem~\ref{thm1} in a special case.  We will show that a minimal sequence can be approximated by an almost product structure after an appropriate blow down. In all of the remaining paper  we use $C$ to denote a  constant that depends only on the dimension of the manifolds.

Let $n\geq m$.
Let $(W^n,g)$ and $(X^m,h)$ be two Riemannian manifolds such that
\[
F: W\to X
\]
is a smooth map. We   assume that $F$ has full rank, that is, ${\rm rank} \,F_*=m$ at any point.  Fixing a point $q\in X$, we assume that $N=F^{-1}(q)$ is a compact connected submanifold in $W$ ($N$ is trivial if $n=m$). By~\cite{CFG}, if the injectivity radius of $W$ is small while the injectivity radius of $X$  is bounded from below,  then $W$ is a fiber bundle over $X$ whose fibers are nilpotent manifolds. Therefore at each point of $X$, there is a neighborhood over which the fiber bundle is trivial.  In Lemma~\ref{lem51}, we prove an  effective version of  the above result.

We introduce some basic notation and assumptions so that we can state the lemma.
Denote by $I\!I$  the second fundamental form of $N$ in $W$.
Let $Rm(W)$ and $Rm(X)$ be the curvature tensors of $W$ and $X$, respectively.

We assume that \begin{equation}\label{new2}
\max |Rm(W)| + \max |Rm(X)|\leq 1.
\end{equation}
Then by ~\cite{CFG}, there is a constant $C_1$ such that
\begin{equation}\label{new1}
|\nabla^2 F| +|I\!I|\leq C_1.
\end{equation}
In addition, we assume that the injectivity radius of $X$ is at least $\sigma>0$.

 Fix an orthonormal basis $\{v_1, \ldots, v_m\}$  of $T_qX$ and let $\{v_1^*, \ldots, v_m^*\}$ denote its lift to $W$  via $F$ in the following sense
\begin{enumerate}[\ \ $(i)$]
\item $F_* v_\alpha^*=v_\alpha$ for $\alpha=1,\cdots, m$;
\item  $v^*_\alpha\perp N$, that is, for any $p\in N$, $(v^*_\alpha)_p\perp T_p N$.
\end{enumerate}
Since ${\rm rank} \,F_*=m$, these lifts  exist and are unique.

Let ${\bf a}=(a_1,\cdots, a_m)\in\R^m$ with $|{\bf a}|<r$ for some $r<\sigma$, the injectivity radius of $X$. We consider the geodesic $\sigma_{\bf{a}}(t)$ in $M$  such that $\sigma_{\bf a} (0) = p\in N\subset W$, and $\sigma'_{\bf a}(0)=\sum a_i v_i^*$.
Define the map $G$ as
\[
G: N\times B_r(T_qX) \to W, \qquad (p, {\bf a})\mapsto \sigma_{\bf a}(1),
\]
where $B_r(T_q(X))$ is the Euclidean ball of radius $r$ in $T_qX$.

\begin{lem} \label{lem51}
We use the above notation, and assume that $|{\bf a}|$ is sufficiently small.
Let $g_W, g_N, g_{\R^m}$ be the corresponding Riemannian metrics over $W, N, \R^m$, respectively. Then
\begin{equation}\label{234}
|G^*g_W-g_N\times g_{\R^m}|_{G^*g_{W}}\leq C|\bf a|.
\end{equation}

\end{lem}

\begin{proof}  Since this is an ``expected'' result, we give a sketch of the proof. If $N$ is a point, then the result of this lemma is well known: it is the closeness of a neighborhood of a manifold to its tangent space within its injectivity radius. In the general case, we need to estimate the effect of the second fundamental form of $N$ on the Jacobi fields.

We can identify the differential $G_*$ as follows: let $x=(p,{\bf a})$ with $p\in N$ and ${\bf a}\in B_r(T_qX)$. Let $w=(\xi,{\bf b})\in T_x(N\times B_r(T_qX))$ with $\xi\in T_pN$ and ${\bf b}=(b_1,\cdots, b_m)\in\R^m$ with $|{\bf b}|<r$. Denote by $\sigma_{\bf a}(t)$ the geodesic starting at $\sigma_{\bf a}(0)=(p,0)$ in the direction $\sigma'_{\bf a}(0)=\sum a_i v_i^*$. We consider the Jacobi field  $V$ along $\sigma_{\bf a}(t)$ such that
\[
V(0)=\xi,\quad V'(0)=\sum b_i v_i^*+\sum a_i\nabla_\xi v_i^*,\quad V''(t)=R(\sigma_{\bf a}'(t), V(t))\sigma_{\bf a}'(t),
\]
where $R(\cdot,\cdot)\cdot$ is the curvature tensor of $M$.
By definition, $G_*w=V(1)$. Therefore, we need to estimate $V(t)$ for $0\leq t\leq 1$.
By ~\eqref{new2} and ~\eqref{new1}, we have
\begin{equation}\label{ineq2}
|\nabla_\xi  v^*|\leq C  \,|\xi|.
\end{equation}
Using this and the upper bound on the curvature tensor,
by the Picard-Lindel\"of Theorem of ODEs, whenever ${\bf a}$ is sufficiently small, we get
\[
\bigl|\langle  G_*w,G_*w\rangle-\langle w,w\rangle_0 \bigr|=\bigl|\langle V(1), V(1)\rangle-\langle w,w\rangle_0 \bigr|\leq C \; |w|^2\cdot|{\bf a}|^2,
\]
where $\langle \cdot,\cdot\rangle_0$ is the inner product with respect to the product metric of $N\times B_r(T_qX)$.  The above inequality is equivalent to~\eqref{234}.
 This proves the lemma.
\end{proof}

The above result allows us to prove Theorem~\ref{thm1}   under some additional assumptions on the minimal sequence.

\begin{thm}\label{cor46}
Let $S=\{(x_i, R_i)\}$ be a minimal sequence such that the injectivity radius at $x_i$ tends to zero as $i\to\infty$.
Moreover, assume that there exists  a sequence of positive real numbers $\eps_i\to0$ such that
\begin{enumerate}
\item $\sqrt{\eps_i} R_i\to 1$ as $i\to \infty$;
\item the curvature of $(B_{x_i}(2\sqrt{\eps_i}R_i), x_i,\eps_i g_i, 2\sqrt{\eps_i} R_i)$ is bounded by $1$;
\item $(B_{x_i}(2R_i), x_i,\eps_i g_i, 2\sqrt{\eps_i} R_i)$  converges to a pointed  smooth manifold $(X,q)$.
\end{enumerate}
Then the essential spectrum of the Laplacian on $k$-forms is either empty, or $[\alpha_k,\infty)$ for some nonnegative number $\alpha_k$.
\end{thm}
\begin{proof}
Let $g_X$ be the Riemannian metric of $X$. By Lemma~\ref{lem51}, by passing to a subsequence if necessary, we may choose a sequence $a_i\to 0$ and nilpotent manifolds $(N_i, g_{N_i})$ such that
\[
|\eps_i g_i-g_{N_i}\times g_X|_{\eps_i g_i}\to 0,
\]
on $(B_{x_i}(a_i), x_i,\eps_i g_i, a_i)$. In particular, we can take $a_i=\eps_i^{1/4}$, then  \[
|g_i-\eps_i^{-1}(g_{N_i}\times g_X)|_{g_i}\to 0,\quad {\rm on }\quad  (B_{x_i}(a_i\eps_i^{-1/2}), x_i,g_i, \eps_i^{-1/2}a_i).
\]

Since $\eps_i^{-1}(g_{N_i}\times g_X)=\eps_i^{-1}g_{N_i}\times \eps_i^{-1}g_X$ and since $X$ is a Riemannian manifold, $\eps_i^{-1}g_X$ is convergent to the Euclidean metric.
Thus by Theorem~\ref{thmSpecGH}, we have
\[
\sigma_{\rm ess}(k,\Delta,M)= [\alpha_k,\infty)=[\lambda_o^{\rm ess},\infty),
\]
where $\alpha_k =\liminf_{i\to\infty} \alpha(N_i,m,n,k)$  and  for the computation of $\alpha(N_i,m,n,k)$ the metric on $N_i$  is $\eps_i^{-1}g_{N_i}$.
\end{proof}

\section{The collapsing case: singular  limit} \label{S5}

In this section we will show that we can always find a minimal sequence which converges to a smooth limit space. To achieve this, we consider the orthonormal frame bundle over an adequate blow-down of the minimal sequence. Our result implies that the essential spectrum is captured by the regular set of the blow-down of the sequence.

We now fix $\epsilon_o>0$ and let $(W,p)$ be a pointed Riemannian manifold, not necessarily complete, such that
\begin{equation}\label{new3}
|Rm(W)|\leq \epsilon_o
\end{equation}
for some $\epsilon_o$   small.
We also   assume  that $W$ is relatively ``large'':
\begin{equation}\label{new3-1}
W \text{ contains a ball of radius } \sqrt{\epsilon_o^{-1}} \text{ centered at } p.
\end{equation}

Define the  $\epsilon_o$-frame bundle $F_{\epsilon_o}(W)$ over $W$ as follows: As a differentiable manifold, $F_{\epsilon_o}(W)$ is the orthonormal frame bundle over $W$. The Levi-Civita connection induces a decomposition of the tangent bundle of $F_{\epsilon_o}(W)$ into the horizontal and vertical sub-bundles. We endow  the horizontal bundle with the pull-back of the Riemannian metric $g_W$ on $W$, and endow the vertical bundle  with $\epsilon_o^{-1}$ times the standard (bi-invariant) Riemannian metric over  $O(n)$. We denote  the Riemannian metric  over $F_{\epsilon_o}(W)$ by $g_{\epsilon_o}=g_{F_{\epsilon_o}(W)}$.

By ~\cite{KS}*{Proposition 2.4}, the curvature of $F_{\epsilon_o}(W)$ satisfies
 \begin{equation}\label{new3-2}
 |Rm(F_{\epsilon_o}(W), g_{\epsilon_o}) |\leq C\epsilon_o,
 \end{equation}
 where $C$ is a constant depending only on the dimension.
Then after rescaling, we have
 \[
 |Rm(F_{\epsilon_o}(W), \epsilon_o g_{\epsilon_o})|\leq C.
 \]
We lift  $p\in W$ to a point $p^*\in F_{\epsilon_o}(W)$, and  consider  the pointed manifold
\[
(F_{\epsilon_o}(W), p^*, \epsilon_og_{\epsilon_o}).
\]
 By ~\cite{Fuk1}*{Theorem 11.1 and Theorem 12.8}, if $\epsilon_o$ is sufficiently small, then there is a pointed \emph{Riemannian manifold} $(Y,y, g_Y)$,  a metric space $X'$, and mappings
 \[
 F: F_{\epsilon_o}(W)\to Y, \quad F_o: W\to X',\quad\pi: Y\to X'
 \]
  such that
\begin{align}\label{438}
&
d_{\rm GH} \left((F_{\epsilon_o}(W), p^*, \epsilon_o g_{\epsilon_o}), (Y,y, g_Y\right))<\Psi(\epsilon_o);\\
& d_{\rm GH}((W,p,\epsilon_o g_W), (X', \pi(y)))<\Psi(\epsilon_o);
\end{align}
and  the group $O(n)$ acts on $F_{\epsilon_o}(W)$  equivariantly:
\begin{equation}\label{256}
\begin{CD}
(F_{\epsilon_o}(W),p^*,\epsilon_o g_{\epsilon_o}) @>F>>(Y,y,g_Y)\\
@VV\pi V@VV\pi V\\
(W,p,\epsilon_o g_{W})  @>F_o >> X'=Y/O(n)
\end{CD}.
\end{equation}
Moreover, the injectivity radius of $Y$ has a positive lower bound.
By~\cite{CFG}, we know that there is an almost flat manifold $N=N_{\epsilon_o}$ such that $F_{\epsilon_o}(W)$ is a fiber bundle with fiber $N$ over $Y$. Without loss of generality, we assume that $W$ is an $A$-regular manifold and hence so is  $F_{\epsilon_o}(W)$ by~\cite{KS}.


Let $\delta>0$ be a small positive number.
In order to  use  Theorem~\ref{cor46}, we
let $W_{\delta}\subset W$ be any  ball of radius $\delta\sqrt{\epsilon_o^{-1}}$ in $W$, and define $Y_{\delta}$ to be the image of $F_{\epsilon_o}(W_{\delta})$ under $F$.   Denote by $\lambda=\lambda_o(k,W_{\delta})$ the smallest eigenvalue of the Friedrichs extension of the Laplace operator on $k$-forms over $W_{\delta}$. Let $\omega$ be an approximate $k$-eigenform for $\lambda$ such that
\begin{enumerate}
\item $\omega$ is smooth with compact support in $W_{\delta}$;
\item for a very small number $\eps$, we have
\[
\|\Delta\omega-\lambda\omega\|_{L^2(W_{\delta},g_W)}\leq\eps\|\omega\|_{L^2(W_{\delta},g_W)}.
\]

\end{enumerate}

Let $\omega^*$ be the pull-back of $\omega$ from $W_{\delta}$ to  $F_{\epsilon_o}(W_{\delta})$ and denote the Laplacian over $F_{\epsilon_o}(W_{\delta})$ by $\Delta^*$. By passing to the universal cover $\tilde W_{\delta}$ of $W_{\delta}$ we know that a cover of $F_{\epsilon_o}(W_{\delta})$ is almost the product\footnote{The universal  cover of $F_{\epsilon_o}(W)$ should in fact be $\tilde W\times Spin(n)$, but here we  obscure the fact a bit to simplify notation.  } $\tilde W_{\delta}\times O_{\epsilon_o}(n)$, where $O_{\epsilon_o}(n)$ is the Riemannian manifold $O(n)$ with the standard Riemannian metric scaled by $\epsilon_o^{-1}$. By~\eqref{new3-2}, $\tilde W_{\delta}\times O_{\epsilon_o}(n)$ is almost Euclidean, and as a result we have the pointwise estimate
\begin{equation}\label{52}
|\Delta^*\omega^*-(\Delta\omega)^*|\leq \Psi(\epsilon_o+\delta)\; \left(|(\nabla^*)^2\omega^*|+|\nabla^*\omega^*|+|\omega^*|\right),
\end{equation}
where $\nabla^*$ denotes the gradient  operator  of $(F_{\epsilon_o}(W_{\delta}), g_{\epsilon_o})$, and $(\Delta\omega)^*$ is the pull-back of $\Delta\omega$ to the frame bundle $F_{\epsilon_o}(W)$.
Since $W$ is an $A$-regular manifold, from ~\eqref{new3-2}, we have
\begin{equation}\label{57}
 |Rm(F_{\epsilon_o}(W), g_{\epsilon_o}) |+ |\nabla Rm(F_{\epsilon_o}(W), g_{\epsilon_o}) |\leq C\epsilon_o.
\end{equation}

Denote ${\mathcal H}=L^2(F_{\epsilon_o}(W_{\delta}), g_{\epsilon_o})$. By the Divergence Theorem and~\eqref{57},
\begin{equation}\label{53}
\|(\nabla^*)^2\omega^*\|^2_{{\mathcal H}}+\|\nabla^*\omega^*\|^2_{{\mathcal H}}+\|\omega^*\|^2_{{\mathcal H}}\leq C\, \left(\|\Delta^*\omega^*\|^2_{{\mathcal H}}+\|\omega^*\|^2_{{\mathcal H}}\right).
\end{equation}
Since
\[
\|(\Delta\omega)^*-\lambda\omega^*\|_{{\mathcal H}}^2={\rm Vol}(O_{\epsilon_o} (n))\cdot \|\Delta\omega-\lambda\omega\|_{L^2(W_{\delta},g_W)}^2\leq\eps^2\|\omega^*\|^2_{\mathcal H},
\]
it follows that
\begin{align*}
\|\Delta^*\omega^*-\lambda\omega^*\|_{{\mathcal H}}^2& \leq 2\|(\Delta\omega)^*-\lambda\omega^*\|_{{\mathcal H}}^2+2
\|\Delta^*\omega^*-(\Delta\omega)^*\|^2_{{\mathcal H}}\\
&\leq2\eps^2\|\omega^*\|^2_{{\mathcal H}}+ 2
\|\Delta^*\omega^*-(\Delta\omega)^*\|^2_{{\mathcal H}}.
\end{align*}
By~\eqref{52},~\eqref{53}, we get
\[
\|\Delta^*\omega^*-(\Delta\omega)^*\|^2_{{\mathcal H}}\leq\Psi(\eps+\epsilon_0+\delta\mid\lambda)\; \left(\|\Delta^*\omega^*-\lambda\omega^*\|_{{\mathcal H}}^2+\|\omega^*\|^2_{\mathcal H}\right).
\]
If $(\eps+\epsilon_o+\delta)$ is small enough,  by  using the above two inequalities, we conclude that
\begin{equation}\label{27-1}
\|\Delta^*\omega^*-\lambda\omega^*\|_{{\mathcal H}}\leq \Psi(\eps+\epsilon_o+\delta\mid \lambda)\;\|\omega^*\|_{{\mathcal H}}.
\end{equation}
In other words, the pull-back of an almost eigenform is also an almost eigenform.

We now describe the set of smooth points of $X'$.

\begin{lem}\label{34}
The set of smooth points $X'_{\rm reg}$ of $X'$ is (Zariski) dense open, and the action of $O(n)$ on the $\pi$-preimage of any smooth point is free.  Moreover, we have the following quantitative version of the lemma: let $X_{\delta}=\pi(Y_{\delta})$. Then there is a point $x'\in X_{\delta}$ such that a ball of radius $C\delta$ centered at $x'$ is contained in $X'_{\rm reg}\cap X_\delta$. The curvature of $x'\in X'_{\rm reg}$ is bounded by $C/d(x')^2$ for a constant $C$, where $d(x')$ is the distance to the singular locus $X_{\rm sing}$ of $X'$.
\end{lem}

\begin{proof} The fact that the set of smooth points of $X'$ is dense open is well-known. We prove that the action of $O(n)$  on any  pre-image of a smooth point of $X'$ is free.

Let   $x'\in X'$ be a smooth point, and let $Z$ be an open ball centered at $x'$  contained in $X'_{\rm reg}$.   Let $W_Z=F_o^{-1}(Z)$.
By shrinking the ball if necessary, we can assume that $W_Z$ is diffeomorphic to $N\times Z$ (cf. ~\cite{CFG}).
Let $F_N$ be the frame bundle of $W$ restricted to $N$. Then $F_{\epsilon_o}(W_Z)$ is diffeomorphic to $F_N\times Z$. Thus as $N$ collapses to a point, we get $\pi^{-1}(Z)=O(n)\times Z$ and the action of $O(n)$ at $x'$ is free.

Now we prove the quantitative version of the lemma. Since $W_{\delta}$ is a ball of radius $\delta\sqrt{\epsilon_o^{-1}}$ and the injectivity radius of $Y$ has a lower bound, by \eqref{438}, without loss of generality, we conclude that the injectivity  radius of $Y_{\delta}=F_{\epsilon_o}(W_{\delta})$ is at least  $\delta$. Therefore,  replacing $Y_\delta$ by the ball of radius $\delta$ if necessary,  we can identify $Y_\delta$ with a small ball of the Euclidean space. Since $X'_{\rm sing}$ is the image, under $\pi$, of the points of $Y_\delta$ on which the action of $O(n)$ is not free, then it is the union of finitely many submanifolds and  it  is closed.
Using  mathematical induction, there exists a point $x'$ such that the ball of radius $C\delta$ (with $C<1$) does not intersect  $X'_{\rm sing}$.

The curvature estimate follows from compactness. Let $x_j\in X'\backslash X'_{\rm sing}$ such that \[d(x_j,X'_{\rm sing})\to 0.\] Assume that $x_j\to x_\infty\in X'_{\rm sing}$. Then, since $(X', x_\infty, d(x_j,X'_{\rm sing})^{-2}g_{X'})$  converges to a tangent cone at $x_\infty$, the curvature estimate follows.
\end{proof}

We shall use the above lemma to study the approximate $k$-eigenform $\omega$, and  its lift  $\omega^*$.
 By Lemma~\ref{lem51},   we know that  $F^{-1}(Y_\delta)$ is diffeomorphic  to $N\times Y_\delta$,  if $\delta$ is sufficiently small,  and we can write
\[
\omega^*=\sum_{s=0}^k  \alpha^s\wedge\beta^{k-s},
\]
where $\alpha^s$ is an $s$-form on $(N,\epsilon_o^{-1}g_N)$,  and $\beta^{k-s}$ is a  $(k-s)$-form on $(Y_\delta, \epsilon_o^{-1} g_Y)$.

Assume that the dimension of $X'$ is $m$.  Then
the dimension of $N$ is $n-m$. We write
\begin{equation}\label{27}
\omega^*=\sum_{k\geq s\geq k-m} \alpha^s\wedge\beta^{k-s}+\sum_{s<k-m} \alpha^s\wedge\beta^{k-s}=\omega_{1}^*+\omega_{2}^*.
\end{equation}

We claim $\omega_2^*\equiv 0$. In order to prove this we will show that $\beta^{k-s}\equiv 0$ for all $s<k-m$. By continuity, it suffices to show that at any point $y'\in Y$ where $\pi(y')$ is smooth, $\beta^{k-s}(y')=0$. To see this, we let $Z$ be a small ball in $X'_{\rm reg}$ centered at $\pi(y')$. Using the notation in Lemma~\ref{34}, we know that if the radius of the ball is sufficiently small, then $W_Z=F_o^{-1}(Z)$ is diffeomorphic to $N\times Z$. In  this case, the commutative diagram~\eqref{256} is reduced to
\[
\begin{CD}
F_N\times Z@>F>> O(n)\times Z\\
@VV\pi V@VV\pi V\\
N\times Z @>F_o >> Z
\end{CD}.
\]
Via the diffeomorphism of  $W_Z$ to $N\times Z$, we can decompose the $k$-form $\omega$   as
\[
\omega=\sum_{k-s\leq m} \xi^s\wedge\eta^{k-s}
\]
where $\xi^s$ is an $s$-form on $N$ and $\eta^{k-s}$ is a $(k-s)$-form on $Z$. We then have
\[
\omega^*=\sum_{k-s\leq m} \pi^*(\xi^s)\wedge\eta^{k-s}.
\]

Comparing the above expression of $\omega^*$ with that in ~\eqref{27}, we   get $\beta^{k-s}\equiv0$ on $\pi^{-1}(Z)$ for $k-m>s$.
Hence
 $\omega_2^*\equiv 0$, and we therefore can write
 \begin{equation}\label{27-2}
 \omega^*=\sum_{k\geq s\geq k-m} \alpha^s\wedge\beta^{k-s}.
 \end{equation}

From ~\eqref{27-1} and assumption ~\eqref{new3} for the curvature tensor, we have
\[
\lambda=\lambda_o(k,W_\delta)\geq \frac{\|\nabla^*\omega^*\|^2_{\mathcal H}}{\|\omega^*\|_{\mathcal H}^2}-\Psi(\eps+\epsilon_o+\delta).
\]
Using ~\eqref{27-2}, together with Lemma~\ref{lem51},
we have
\[
\frac{\|\nabla^*\omega^*\|^2_{\mathcal H}}{\|\omega^*\|_{\mathcal H}^2}\geq \min_{k\geq s\geq k-m}\left[ \frac{\|\nabla_N \alpha^s\|^2_{L^2(N,\epsilon_o^{-1}g_N)}}{\|\alpha^s\|_{L^2(N,\epsilon_o^{-1}g_N)}^2}+\frac{\|\nabla_Y \beta^{k-s}\|^2_{L^2(Y_{\delta},\epsilon_o^{-1}g_Y)}}{\|\beta^{k-s}\|_{L^2(Y_{\delta},\epsilon_o^{-1}g_Y)}^2}\right]-\Psi(\epsilon_o+\delta).
\]
Eliminating the Rayleigh quotient for $Y$,   we get
\[
\frac{\|\nabla^*\omega^*\|^2_{\mathcal H}}{\|\omega^*\|_{\mathcal H}^2}\geq \min_{s\geq k-m}\lambda_o(s, (N, \epsilon_o^{-1} g_N))-\Psi(\epsilon_o+\delta\mid\lambda).
\]
We can verify that
\[
\alpha((N,\epsilon_o^{-1}g_N), m,n,k)=\min_{k\geq s\geq k-m}\lambda_o(s,(N, \epsilon_o^{-1} g_N)),
\]
noting that for $k<m,$ $\lambda_o(0,(N, \epsilon_o^{-1} g_N))=0$.  Thus we get
\begin{equation}\label{2c2}
\lambda\geq \alpha((N,\epsilon_o^{-1}g_N), m,n,k)-\Psi(\eps+\epsilon_o+\delta\mid\lambda).
\end{equation}

Now we can prove the main result of this section  which, as we will see in the next section, allows us to replace the minimal sequence by one that collapses to a smooth manifold.

\begin{thm}\label{thm522}
Let $(W,p)$ be a pointed Riemannian manifold which satisfies ~\eqref{new3},~\eqref{new3-1}. Then for any $\delta>0$  small enough, there is an open subset $W'\subset W$ such that
\begin{enumerate}
\item $F_o(W')\subset X'_{\rm reg}$;
\item $W'$ is  a ball of radius $C\delta\sqrt{\epsilon_o^{-1}}$ for some constant $C$;
\item $\lambda_o(k, W')\leq \lambda_o(k, W)+\Psi(\epsilon_o+\delta\mid \lambda_o(k, W))$.
\end{enumerate}
\end{thm}

\begin{proof}
We take a ball $Z\subset X'_{\rm reg}\cap X_\delta$ of radius $C\delta\sqrt{\epsilon_o^{-1}}$, and define
$W'=W_Z$. Then, by  Lemma~\ref{34} $W'$ is almost a product $N\times Z$ with the metric $\epsilon_o^{-1} g_N\times \epsilon^{-1}_o g_{X'}$.

Let $d(x')$ be the distance of a point $x'\in X'$ to the singular locus $X'_{\rm sing}=X'\backslash X'_{\rm reg}$. Then the curvature of $X'$ at $x'$ is bounded by $C/(d(x'))^2$.  If we shrink the radius of $Z$ by half, then the curvature of $X'$ on $Z$, with respect to the metric $\epsilon_o^{-1} g_{X'}$, is bounded by $C\delta^{-2}\epsilon_o$.  Therefore, for a fixed $\delta$, if $\epsilon_o=\epsilon_o(\delta)$ is small enough, then the curvature goes to zero.
By Theorem~\ref{thmSpecGH}, we have
\[
\lambda_o(k,W')\leq \alpha((N,\epsilon_o^{-1}g_N), m,n,k)+\Psi(\epsilon_o(\delta)+\delta)=\alpha((N,\epsilon_o^{-1}g_N), m,n,k)+\Psi(\delta).
\]
By ~\eqref{2c2}, we have
\[
\lambda_o(k,W')=\alpha((N,\epsilon_o^{-1}g_N), m,n,k)+\Psi(\delta)\leq \lambda_o(k, W)+ \Psi(\delta\mid \lambda_o(k, W)).
\]
This completes the proof of the theorem.
\end{proof}

\section{Proof of Theorem~\ref{thm1}} \label{S6}
In this section, we prove Theorem~\ref{thm1}. Let $\eps_i\to 0$, and $R_i>C \eps_i^{-2}$. By Proposition~\ref{prop26}, there exists a minimal sequence $S=\{(x_i, R_i)\}$ such that the balls $M_i=B_{x_i}( 2 R_i)$ are disjoint and the curvature of $M_i=B_{x_i}(R_i)$ is bounded by $\eps_i$.  Let $I(x_i)$ be the injectivity radius at $x_i$.

\begin{proof}[Proof of Theorem \ref{thm1}]  First, we assume that $\liminf I(x_i)>\nu$ for some $\nu>0$.  By \cite{Fuk1}*{Theorem 6.7} (see also  \cite{Gr1}*{(8.20)}),  whenever we have a convergent sequence of pointed manifolds $(M_i,x_i)$ with injectivity radius uniformly bounded below and bounded curvature, the limit space $(X,p)$ is a smooth manifold with a $\mathcal C^{1,\alpha}$-metric tensor.  Since we are able to assume that the  $M_i$ are $A$-regular,  the convergence is in the ${\mathcal C}^\infty$-topology, and the limit space $(X,p)$ is a smooth manifold.

In fact, the limit space $(X,p)$ is a flat and complete noncompact manifold.
As a result, there are mappings $F_i: B_{x_i}(R_i)\to X$ which are diffeomorhisms of  the  balls $B_{x_i}(R_i)$  onto their images. Moreover  the  pullback metrics  $(F_i^{-1})^*g_{M_i}$ are  convergent to the flat metric of $X$ in the $\mathcal C^\infty$-topology.

As we saw in Theorem~\ref{thmF2}, the essential spectrum of $X$ is a connected interval. Therefore, for any $\lambda\in\sigma_\mathrm{ess}(k,\Delta, X)$, working as in the proof of Theorem \ref{thmSpecGH}, we can use the pull-backs of  the approximate eigenforms on $X$ to $M$ and conclude that the essential spectrum of $M$ is  a connected interval as well.

Note that in the case $I(x_i) \to \infty$ the limit space $X$ is the Euclidean space $\mathbb{R}^n$. It follows that in this case the essential spectrum on $k$-forms is $[0,\infty)$ for all $k$.

We now assume that $I(x_i)\to 0$.  This is the most complicated case of the two, and requires full use of the results from the previous two sections.  Let $\delta_i=(\eps_i)^{1/4}$.
 By Theorem~\ref{thm522}, we can find another sequence $M_i'\subset M_i$, and $F_i: M_i\to X'$ such that
\begin{enumerate}
\item $F_i(M_i')\subset X'_{\rm reg}$;
\item $M_i'$ is a ball of radius $R_i'=C(\eps_i)^{-1/4}$;
\item $\lambda_o(k, M_i')\leq \lambda_o(k, M_i)+\Psi(\eps_i\mid\lambda_o(k, M_i))$.
\end{enumerate}
Without loss of generality, we may assume that $M_i'=B_{y_i}(R_i')$, where $y_i\in M_i$ and $R_i'\to\infty$. Note that the $M_i'$ are also disjoint.
Then the 3rd condition above  implies that   $S'=\{(y_i, R_i')\}$ is a minimal sequence. Since $S'$  satisfies the conditions of Theorem~\ref{cor46}, we conclude  that the essential spectrum is either empty or a connected interval.
\end{proof}

Let $\tilde \Delta$ denote the covariant Laplacian on $k$-forms.  By the Weitzenb\"ock formula and  asymptotic flatness, we have

\begin{corl}
Let $(M^n,g)$ be a complete noncompact Riemannian manifold which is asymptotically flat.  Then for $0\leq k\leq n$,  $\sigma_\mathrm{ess}(k,\tilde \Delta, M)=\sigma_\mathrm{ess}(k,  \Delta, M)$  and is therefore either empty or $[\alpha_k,\infty)$ for a nonnegative number $\alpha_k$, in other words it is computable.
\end{corl}

\section{The $k$-form Spectrum over Manifolds with Asymptotically Nonnegative Ricci Curvature}\label{10}

In this final section we prove Theorem \ref{thmRic} and Corollary \ref{cor14}. We first give the following definition for   the  dimension of a manifold at infinity  following Cheeger and Colding.

\begin{Def}[Dimension at infinity]\label{77} Suppose that $M$ has asymptotically nonnegative Ricci curvature along a sequence of expanding balls as in Definition \ref{76}. We consider any positive sequence $\eps_i\to 0$ such that $\sqrt{\eps_i} R_i\to\infty$, but $\eps_i^{-1}\delta_i\to 0$. Then, after possibly passing to a subsequence, $(B_{x_i}(R_i), x_i, \eps_i g, \sqrt{\eps_i} R_i)$
is convergent in the pointed Gromov-Hausdorff sense to a metric  space  $(X,p)$.   We define the supremum of all the dimensions of the possible limit spaces as the dimension at infinity of $M$, and we denote it by $\dim_\infty M$.
\end{Def}
If the Ricci curvature of manifold is nonnegative, then $\dim_\infty M$ is the dimension of the cone of the manifold at infinity. The following  fact is straightforward based on the results in~\cite{CCoIII}.

\begin{lem} \label{lemGH}
Under the assumptions of Theorem~\ref{thmRic},
there exists a sequence $M_i=B_{x_i}(R_i)$ and $p_i\in M_i$, a  sequence $\eps_i\to 0$, and a sequence $R_i'\to\infty$ such that
\[
d_{\rm GH}\left((M_i, p_i, \eps_i g, R_i'),(\mathbb{R}^q, 0, g_E, \infty)\right) \to 0
\]
as $i\to \infty$,   where $q\geq 1$ is the dimension of $M$ at infinity, and $g_E$ is the standard Euclidean metric of $\R^q$.
\end{lem}

The above lemma allows us to find $L^2$ approximate eigenfunctions for every  spectral point $\lambda \in[0,\infty)$  for manifolds with asymptotically nonnegative Ricci curvature. This was not previously possible, even for manifolds with Ricci curvature nonnegative. It also provides a novel proof of the same result due to the second author and D. Zhou~\cite{Lu-Zhou_2011}   without resorting to Sturm's $L^p$ independence result \cite{sturm}.

\begin{lem} \label{lemGH2}
Let $(M,g)$  be a complete noncompact Riemannian manifold. Assume that on a sequence of geodesic balls $M_i=B_{x_i}(R_i)$ with $R_i\to \infty$ there exist $\delta_i \to 0$ such that $\mathrm{Ric}_{M_i} \geq - \delta_i.$  Then for each $\lambda\geq 0$ and each $i$ there exists a smooth function  $\phi_i$ whose support  lies    in $M_i$ such that
\[
\|(\Delta - \lambda) \phi_i\|_{L^2(M)}  \leq \Psi(i^{-1}) \|\phi_i\|_{L^2(M)} .
\]
As a result, the spectrum of the  Laplacian on functions over $M$ is $[0,\infty)$.
\end{lem}

\begin{proof}
By  Lemma \ref{lemGH},  there exist a sequence $x_i\in M_i$ and $\eps_i\to 0$ such that
\[
d_{\rm GH}\left((M_i, x_i, \eps_i g, 5),(\mathbb{R}^q, 0, g_E, 5)\right) \to 0.
\]

By \cite{CCoII}, there exist harmonic maps
\[
\Phi_i: (M_i, x_i, \eps_i g, 3) \to \mathbb{R}^q
\]
satisfying the following properties.  First, we have  $\Phi_i(M_i, x_i, \eps_i g, 1) \subset (\mathbb{R}^q, 0, g_E, 2)$ and $\; |\nabla  \,\Phi_i| \leq c(n)\;$. Writing the map $\Phi_i$ in coordinates, $\Phi_i = (b_{i,1}, \cdots, b_{i,q})$, then for $1\leq j,l\leq q$,  we have
\begin{align*}
&\int_{(M_i, x_i, \eps_i g, 1)} \sum_{j,l} |\mathrm{Hess} \,b_{i,j}|^2 \leq \Psi(i^{-1}) \cdot \mathrm{Vol} ((M_i, x_i, \eps_i g, 1)), \quad and \\
&\int_{(M_i, x_i, \eps_i g, 1)} \sum_{j,l}  |\<\n b_{i,j}, \n b_{i,l}\>-\delta_{jl}|\leq\Psi(i^{-1}) \cdot\mathrm{Vol} ((M_i, x_i, \eps_i g, 1)).
\end{align*}
The  Gromov-Hausdorff convergence implies that
\[
|d(\Phi_i(q_1),\Phi_i(q_2)) - d(q_1,q_2)| \leq \Psi(i^{-1}) \]
for any $q_1,q_2\in (M_i, x_i, \eps_i g, 1)$,
where the distances are with respect to the metrics $\eps_i g$ and the Euclidean distance respectively.

In order to use the above results, we make the following rescaling.
We use the diffeomorphism $\xi_i:(M_i, x_i, g, \eps_i^{-1/2}) \to (M_i, x_i, \eps_i g, 1)$ and a similar map on $(\mathbb{R}^q, 0, g_E, 1)$ to get rescaled  maps
\[
\tilde\Phi_i: (M_i, x_i,   g,  \eps_i^{-1/2}) \to  (\mathbb{R}^q, 0, \eps_i^{-1} g_E, \eps_i^{-1/2}),\quad \quad\tilde\Phi_i =\frac{\Phi_i}{\sqrt\eps_i}= (\tilde b_{i,1}, \cdots, \tilde b_{i,q}), \\
\]
which are also harmonic. It is clear  that the rescaled maps also satisfy $|\nabla \tilde \Phi_i|\leq c(n)$ and
\begin{align} \label{ric_2}
\begin{split}
&\int_{(M_i, x_i, g, {\eps_i}^{-1/2})} \sum_{j,l} |\mathrm{Hess}\, \tilde b_{i,j}|^2 \leq \Psi(i^{-1})\cdot \mathrm{Vol} ((M_i, x_i, g, \eps_i^{-1/2})),\\
&\int_{(M_i, x_i, g, {\eps_i}^{-1/2})} \sum_{j,l}  |\<\n \tilde b_{i,j}, \n \tilde b_{i,l}\>-\delta_{jl}|\leq \Psi(i^{-1})\cdot\mathrm{Vol} ((M_i, x_i, g, \eps_i^{-1/2})).
\end{split}
\end{align}
For any $q_1, q_2 \in (M_i, x_i, g, {\eps_i}^{-1/2})$, we have
\begin{equation}\label{ric_3}
|d(\tilde\Phi_i(q_1),\tilde \Phi(q_2)) - d(q_1,q_2)| \leq \Psi(i^{-1}) \eps_i^{-1/2}  \end{equation}
where $d(x,y)$ denotes the distance in the (rescaled) respective metrics. In particular, if we take $q=q_1$, and $\tilde\Phi(q_2)=0$, we have
\begin{equation}\label{ric_4}
\bigl|{|\!|}\tilde \Phi_i{|\!|}(q)-d(q,q_2)\bigr|= \Psi(i^{-1}) \eps_i^{-1/2}
\end{equation}
where $  {|\!|} \tilde \Phi_i{|\!|}=\sqrt{\tilde b_{i,1}^2+\cdots+\tilde b_{i,q}^2}$.

Let $\chi_i(r)$ be  a smooth function on $\R$ such that
$\chi_i(r)=1$ for $|r|\leq \eps_i^{-1/2}/2$; $\chi_i(r)=0$ for $|r|>\eps_i^{-1/2}$; and $|\chi_i'|\leq 4\sqrt{\eps_i}$.
 Since $\tilde\Phi_i$ is harmonic, a straightforward computation shows that
\[
\bigl||\n {|\!|}\tilde \Phi_i{|\!|}|^2-1\bigr|\leq C\sum_{j,l}  |\<\n \tilde b_{i,j}, \n \tilde b_{i,l}\>-\delta_{jl}|,\quad  \bigl|\Delta {|\!|}\tilde \Phi_i{|\!|}\bigl|\leq C\; \frac{\sum_{j,l}  |\<\n \tilde b_{i,j}, \n \tilde b_{i,l}\>-\delta_{jl}|}{{|\!|}\tilde \Phi_i{|\!|}}.
\]
Define
\[
\phi_i=e^{i\sqrt\lambda \,\tilde b_{i,1}}\chi_i({|\!|}\tilde \Phi_i{|\!|}).
\]
By~\eqref{ric_4}, we know that $\phi_i$ is a smooth function with compact support.
For $\eps_i^{-1/2}/2<{|\!|}\tilde \Phi_i{|\!|}<\eps_i^{-1/2}$, we have
\[
|(\Delta - \lambda) \phi_i|^2\leq C  \left| |\nabla \tilde b_{i,1}|^2-1\right|+C\sqrt{\eps_i},
\]
and for all other values of $\|\tilde \Phi_i\|$ the left side is zero.   Estimate~\eqref{ric_2} gives
\[
\|(\Delta - \lambda) \phi_i\|^2_{L^2(M)} \leq \Psi(i^{-1}) \cdot\mathrm{Vol} ((M_i, x_i, g, \eps_i^{-1/2})).
\]
By~\eqref{ric_4},
\[
\{q\mid {|\!|}\tilde\Phi_i{|\!|}(q)<\eps_i^{-1/2}/2\}\supset B_p((1-\Psi(i^{-1}) )\eps_i^{-1/2}/2).
\]
 Therefore,
\[
\|\phi_i\|^2_{L^2(M)} \geq \mathrm{Vol} ((M_i, x_i, g, (1-\Psi(i^{-1}) )\eps_i^{-1/2}/2)).
\]
Since the Ricci curvature is asymptotically nonnegative, we have the following volume comparison inequality
\begin{equation}\label{vol_3}
\mathrm{Vol} ((M_i, x_i, g, \eps_i^{-1/2}))\leq C\, \mathrm{Vol} ((M_i, x_i, g, (1-\Psi(i^{-1}) )\eps_i^{-1/2}/2)).
\end{equation}

The lemma is proved by combining the above inequalities.
\end{proof}

Finally, we are able to prove  Theorem~\ref{thmRic}, which is an extension of Lemma~\ref{lemGH2}.
\begin{proof}[Proof of Theorem~\ref{thmRic}] Assume that $k\leq q$.
Define the test $k$-forms
\[
\omega_i =\phi_i  \; d\tilde b_{i,1} \wedge \cdots \wedge d\tilde b_{i,k}
\]
where  $\phi_i$ and $\tilde b_{i,j}$ are defined as in Lemma~\ref{lemGH2}. By the Weitzenb\"ock formula, it is well known that
\begin{equation} \label{ric_1}
\begin{split}
(\Delta  - \lambda) \omega_i = & (\Delta \phi_i - \lambda \phi_i \,)  \, d\tilde b_{i,1} \wedge \cdots \wedge d\tilde b_{i,k} -2 \n_{\n \phi_i } (d\tilde b_{i,1} \wedge \cdots \wedge d\tilde b_{i,k}) \\
& + \phi_i \,\Delta  (d\tilde b_{i,1} \wedge \cdots \wedge d\tilde b_{i,k}).
\end{split}
\end{equation}

Fix any $\alpha >1$. Using formula \eref{ric_1}  together with the properties of the $\tilde b_{i,j}$ we get, for $\gamma=1,2$, that
\begin{align} \label{ric_30}
\begin{split}
&\bigl|(\,(\Delta +\alpha)^{-\gamma}\omega_i,(\Delta  - \lambda) \omega_i \,)\bigr| \\ &\leq C \,\| \phi_i\|_{L^2(M)}  \cdot  \left[
\|(\Delta - \lambda) \phi_i   \|_{L^2(M)}   + \| |\n \phi_i|\cdot |\mathrm{Hess}\,\tilde\Phi_i| \, \|_{L^2(M)}   \right] \\
&  + \bigl|(\,(\Delta +\alpha)^{-\gamma}\omega_i,\phi_i \Delta (d\tilde b_{i,1} \wedge \cdots \wedge d\tilde b_{i,k}) \,)\bigr|,
\end{split}
\end{align}
where we have used the fact that $(\Delta +\alpha)^{-\gamma}$ is bounded on $L^2$ for the first two terms in the right side.

By the proof of  Lemma~\ref{lemGH2} and ~\eqref{ric_2}, we have
\begin{equation}\label{ric_10}
\|(\Delta - \lambda) \phi_i   \|_{L^2(M)}   + \| |\n \phi_i|\cdot |\mathrm{Hess}\,\tilde\Phi_i| \, \|_{L^2(M)}  \leq
\Psi(i^{-1})\cdot\sqrt{\mathrm{Vol} ((M_i, x_i, g, \eps_i^{-1/2}))}.
\end{equation}
For the third term in the right side of \eref{ric_30}, we let $\eta_i = (\Delta +\alpha)^{-\gamma}\omega_i$ and observe that
\begin{equation*}
|(\,\eta_i ,\phi_i \Delta (d\tilde b_{i,1} \wedge \cdots \wedge d\tilde b_{i,k}) \,)| = |(\, \delta(\phi_i \eta_i) , \delta(d\tilde b_{i,1} \wedge \cdots \wedge d\tilde b_{i,k}) \,)|.
\end{equation*}
Since $ \delta(\phi_i \eta_i) = -\iota(\n \phi_i) \eta_i + \phi_i \delta \eta_i $, and
\begin{equation*}
\begin{split}
\|  \delta \eta_i \|_{L^2(M)}^2 & \leq \|  \delta \eta_i \|_{L^2(M)}^2 + \|  d \eta_i \|_{L^2(M)}^2  = (\eta_i, \Delta (\Delta +\alpha)^{-\gamma}\omega_i) \\
& =  (\eta_i,  (\Delta +\alpha)^{-\gamma+1}\omega_i) - \alpha (\eta_i,  (\Delta +\alpha)^{-\gamma}\omega_i ) \leq C\| \omega_i\|_{L^2(M)}^2,
\end{split}
\end{equation*}
we have $ \|\delta (\phi_i\eta_i) \|_{L^2(M)}\leq C\|\omega_i\|_{L^2(M)}. $ Therefore,  by \eref{ric_2}
\begin{align}  \label{ric_6}
\begin{split}
&|(\, \delta(\phi_i \eta_i) , \delta(d \tilde b_{i,1} \wedge \cdots \wedge d \tilde b_{i,k}) \,)|\\ & \leq C \|\omega_i\|  \;   \|\mathrm{Hess} \,\tilde \Phi_i\|  \leq \Psi(i^{-1}) \|\omega_i\|_{L^2(M)}\cdot\sqrt{\mathrm{Vol} ((M_i, x_i, g, \eps_i^{-1/2}))}.
\end{split}
\end{align}
In order to estimate the $L^2$-norm of $\omega_i$ from below, we use the following fact from linear algebra
\[
|d\tilde b_{i,1} \wedge \cdots \wedge d\tilde b_{i,k}|^2\geq 1-c(n)\sum_{j,l}|\<\n \tilde b_{i,j}, \n \tilde b_{i,l}\>-\delta_{jl}|.
\]
Thus  for $i$ sufficiently large we obtain
\begin{equation} \label{ric_9}
\begin{split}
\|\omega_i\|^2 &\geq \int_{{|\!|}\tilde\Phi_i{|\!|}<\eps_i^{-1/2}/2} |d\tilde b_{i,1} \wedge \cdots \wedge d\tilde b_{i,k}|^2\geq
C\, \mathrm{Vol} ((M_i, x_i, g, (1-\Psi(i^{-1}) )\eps_i^{-1/2}/2)).
\end{split}
\end{equation}

Using \eref{ric_30}, \eref{ric_10}, ~\eqref{ric_6} and ~\eqref{ric_9}, together with the volume comparison inequality ~\eqref{vol_3}, we get
\[
\bigl|(\,(\Delta +\alpha)^{-\gamma}\omega_i,(\Delta  - \lambda) \omega_i \,)\bigr|  \leq \Psi(i^{-1})  \,\| \omega_i\|^2
\]
for $\gamma=1,2$. Then by  the generalized Weyl Criterion (Corollary~\ref{cor21}), $\lambda\in\sigma_{\rm ess}(k,\Delta, M)$ for $k\leq q$. The case $k\geq n-q$ follows from Poincar\'e Duality.
\end{proof}

\begin{remark}
Contrary to Lemma~\ref{lemGH2},
the Cheeger-Colding estimates do not appear sufficient to prove Theorem \ref{thmRic} using the classical Weyl criterion  for the case of $k$-forms with $1\leq k\leq n/2$. This is due to the fact that they do not provide adequate $L^2$-estimates for the third term in the right side of \eref{ric_1}. Our generalized Weyl criterion allowed us to simplify this term by considering instead the third term in \eref{ric_30} and integrating by parts. \end{remark}

Since $q\geq 1$, the result of Theorem \ref{thmRic} is true over a manifold with asymptotically nonnegative Ricci curvature for the case of functions and 1-forms and is therefore consistent with \cite{ChLu6} where it was proved that the 1-form spectrum always contains the function spectrum.  Our result partially addresses the $k$-form spectrum of these manifolds for $k>1$, which was until now completely open.

 \begin{proof}[Proof of Corollary~\ref{cor14}]
It is well known that if the volume of $M$ satisfies ~\eqref{c5}, then the dimension of the manifold at infinity is at least $s$. The corollary then follows from
Theorem~\ref{thmRic}.
 \end{proof}

\section{Examples and Further Discussions}\label{11}

In the section we include a few interesting examples of manifolds where we can compute the essential spectrum on $k$-forms explicitly, and then discuss Conjecture \ref{16}.

\begin{example}
Let $(K_o^3,g_o)$ denote the compact flat three-manifold constructed by Hantzsche and Wendt in 1935 with first Betti number zero. Consider the warped product manifold $M^4= \mathbb{R} \times_{\varphi} K_o^3$ with metric $dr^2 +\varphi(r)^2 g_o$ and $\varphi(r)= |r|^{-1}$ for $|r|\geq r_o$.  Then, the sectional curvature, and hence the curvature tensor, of $M$ is asymptotically zero  satisfying  $|Rm(M)|\leq c/r^2$ for $|r|\geq 2r_o$.

By \cite{char-lu-1}*{Theorem 1.3}, $\sigma_{\mathrm ess} (0,\Delta)=[0,\infty)$. Using \cite{ChLu6}*{Theorem 4.1} we also get $\sigma_{\mathrm ess} (1,\Delta)=[0,\infty)$. Poincar\'e duality gives
\[
\sigma_{\mathrm ess} (k,\Delta)=[0,\infty) \qquad \text{for} \ \ k=0,1,3,4.
\]

At the same time, it is easy to see that any minimal sequence as obtained in Proposition \ref{prop26} collapses in the pointed Gromov-Hausdorff sense to $\mathbb{R}$. Moreover, any minimal sequence is close to a product manifold as in Theorem \ref{thmSpecGH} with fiber $N_i=K_o^3$ which shrinks as $i\to \infty$. As a result,
\[
\sigma_{\mathrm ess} (2,\Delta)= \emptyset.
\]

This result is also consistent with what one can get by decomposing the Laplacian on 2-forms over a manifold with a pole as in \cite{Don2}*{Section 4}. In this case, due to the positive first eigenvalue on 1-forms and 2-forms of the cross-section $K_o^3$, the Laplacian on 2-forms reduces to second order operators on the real line whose potential becomes infinite, and hence have only discrete spectrum.
\end{example}

\begin{example}
Consider now a more general warped product manifold $M^n= \mathbb{R} \times_{\varphi} K^{n-1}$ where $ K^{n-1}$ is a compact flat manifold. We endow $M$ with the metric $dr^2 +\varphi(r)^2 g_o$ and $\varphi(r)= |r|^{-1}$ for $|r|\geq r_o$.  Let $b_k$ denote the $k^{\text{th}}$ Betti number of $K$. Again the curvature tensor of $M$ is asymptotically zero. Working as in the previous example we get that for $1\leq k \leq n/2$
\begin{equation*}
\begin{split}
&\sigma_{\mathrm ess} (k,\Delta)= \emptyset \qquad \text{whenever} \qquad b_{k-1}=b_k=0,  \qquad \text{and}\\
&\sigma_{\mathrm ess} (k,\Delta)= [0,\infty) \qquad \text{otherwise}.
\end{split}
\end{equation*}

\end{example}

We now consider doubly warped product manifolds of the type $M=I \times_{\varphi} K_1^p \times_{\psi} K_2^q$ where $I\subset \mathbb{R}$ for $r\in I$, and $(K_1^p,g_1)$ and $(K_2^q,g_2)$ are compact Riemannian manifolds of dimension $p$ and $q$ respectively. The functions $\varphi(r)$, $\psi(r)$ define the doubly warped metric on $M$, which is given by
\[
g=d\rho^2 + \varphi(\rho)^2 g_1 + \psi(\rho)^2 g_2
\]
In the case when the manifolds $K_1$, $K_2$ are flat, then following the same procedure as in \cite{Pet}*{Chapter 3} the sectional curvatures of $M$ are convex linear combination of
\[
-\frac{\varphi''}{\varphi}, \  -\frac{\psi''}{\psi}, \ -\frac{(\varphi')^2}{\varphi^2}, \ -\frac{(\psi')^2}{\psi^2}, \ -\frac{\varphi'\psi'}{\varphi \psi}.
\]
If on the other hand $K_1$ is a $p$-dimensional sphere $S^p$, then the third term in the list above must be replaced by $ \left( 1- (\varphi')^2 \right)/\varphi^2$, and similarly for the fourth term if $K_2$ becomes $S^q$.

Using a doubly warped product we can also construct manifolds for which
\begin{equation*}
\begin{split}
&\sigma_{\mathrm ess} (k,\Delta)=[\alpha_k,\infty) \qquad \text{for} \ \ 0\leq k \leq m <n/2, \qquad \text{and}\\
&\sigma_{\mathrm ess} (k,\Delta)=\emptyset \qquad \text{for} \ \ m <k \leq n/2
\end{split}
\end{equation*}
for some $m$, with $0\leq \alpha_k<\infty$. One such example is the following.

\begin{example}
With $K_o^3$ as in the previous example, we construct the 13-dimensional manifold
\[
M^{10}= \mathbb{R} \times_{\varphi} K_o^3 \times_{\psi} (K_o^3\times K_o^3\times K_o^3)
\]
with $\varphi(r)= C_o$, a constant for all $r$, and $\psi(r)=1/r$ for $|r|\geq r_o$. Then  as  in the previous example, the sectional curvature, and hence the
curvature tensor, of $M$ is again asymptotically zero satisfying  $|Rm(M)|\leq c/r^2$.

Then
\begin{equation*}
\begin{split}
& \sigma_{\mathrm ess} (k,\Delta)=[0,\infty) \qquad \text{for} \ \ k=0,1,3,4 \\
& \sigma_{\mathrm ess} (2,\Delta)=[\alpha_2,\infty), \qquad \text{and} \ \ \\
& \sigma_{\mathrm ess} (k,\Delta)=\emptyset \qquad \text{for} \ \ k=5,6,
\end{split}
\end{equation*}
where $\alpha_2$ is the first  eigenvalue of the Laplacian on 1-forms over $K_o$. The result for the spectrum on 5-forms and 6-forms, follows from the fact that any minimal sequence collapses in the pointed Gromov-Hausdorff sense to a 4-dimensional manifold, and there are no harmonic 5-forms, nor 6-forms on $M$. For $7\leq k\leq 13$, we get the spectrum by Poincar\'e duality.

\end{example}

We end this paper by discussing Conjecture~\ref{16}, where we claimed that whenever Ricci curvature is asymptotically non-negative, then the spectrum of the Laplacian on differential $k$-forms is \emph{computable} for all $k$. For the function spectrum, the conjecture holds (cf.~\cites{Lu-Zhou_2011,wang}). By ~\cite{ChLu6} and Poincar\'e duality, for $k=1,n-1,n$,  $\sigma_{\mathrm ess} (k,\Delta)=[0,\infty)$ as well.  For the general case, we first make the following observation.

\begin{prop}\label{84}
Let $\lambda\in\R$ be a non-negative real number, and $B_{p_i}(R_i)$ a sequence of disjoint balls  with $R_i\to\infty$. Let  $\omega_i$  be  a $k$-form $B_{p_i}(R_i)$  such that $\Delta\omega_i=\lambda\omega_i$ (note that no boundary conditions on $\omega_i$ are assumed).
Then, $\lambda\in \sigma_{\mathrm ess} (k,\Delta)$ if there exists a sequence $R_i'<R_i-3$, $R_i'\to\infty$, such that
\begin{equation} \label{e81}
(R_i-R_i')^{-2}\int_{B_{p_i}(R_i)\backslash B_{p_i}{(R_i')}}|\omega_i|^2\leq \Psi(i^{-1})\int_{B_{p_i}(R_i')}|\omega_i|^2.
\end{equation}
\end{prop}

\begin{proof} For any smooth $k$-form $\omega$, and smooth function $\rho$ we have the identity
\[
\Delta (\rho\omega)-\rho\Delta\omega=d\rho\wedge\delta\omega-\iota(d\rho) d\omega
-d(\iota(d\rho)\omega)+\delta (d\rho\wedge\omega)
\]
where $\iota$ denotes the contraction operator. Therefore, for $\gamma\geq 1$ we get
\begin{align*}
((\Delta-\lambda)(\rho\omega)&,(\Delta+1)^{-\gamma}(\rho\omega))=\\
=&(\rho(\Delta-\lambda)\omega,(\Delta+1)^{-\gamma}(\rho\omega)) +((d\rho\wedge\delta-\iota(d\rho)d)\omega, (\Delta+1)^{-\gamma}(\rho\omega))\\
&-(\iota(d\rho)\omega, \delta(\Delta+1)^{-\gamma}(\rho\omega))+(d\rho\wedge\omega, d(\Delta+1)^{-\gamma}(\rho\omega)).
\end{align*}
Let $\rho$ be a cut-off function such that $\rho\equiv 1$ on $B_{p_i}(R_i'+1)$ and $\rho\equiv 0$ outside $B_{p_i}(R_i-1)$. Consider the annulus $A_i=B_{p_i}(R_i-1)\backslash B_{p_i}(R_i'+1)$. Setting $\omega=\omega_i$, since $(\Delta-\lambda)\omega=0$ on $A_i$, we have
\begin{align*}
|((\Delta-\lambda)(\rho\omega)&,(\Delta+1)^{-\gamma}(\rho\omega))|\\
&\leq C(R_i-R_i')^{-1}\,(\|\delta\omega\|_{L^2(A_i)}+\|d\omega\|_{L^2(A_i)}+\|\omega\|_{L^2(A_i)})\cdot \|\rho\omega\|_{L^2(M)}.
\end{align*}

Let $\tilde\rho$ be another cut-off function such that $\tilde \rho\equiv 1$ on $A_i$;  $\tilde \rho$ vanishes outside
the annulus $B_{p_i}(R_i)\backslash B_{p_i}{(R_i')}$; and $|\nabla \tilde\rho|\leq 2$. Using integration by parts and Young's  inequality, we get
\[
\|\delta\omega\|_{L^2(A_i)}^2+\|d\omega\|_{L^2(A_i)}^2\leq
\|\tilde\rho^2\delta\omega\|_{L^2}^2+\|\tilde\rho^2d\omega\|_{L^2}^2\leq
 C \|\omega\|^2_{L^2(B_{p_i}(R_i)\backslash B_{p_i}{(R_i')})}.
\]
Therefore,
\begin{equation}\label{347}
|((\Delta-\lambda)(\rho\omega),(\Delta+1)^{-\gamma}(\rho\omega))|\leq   C(R_i-R_i')^{-1}\|\omega\|_{B_{p_i}(R_i)\backslash B_{p_i}{(R_i')})}\cdot \|\rho\omega\|_{L^2(M)}.
\end{equation}
The proposition  follows by Corollary~\ref{cor21}.
\end{proof}

\begin{corl}\label{85} Let $p\in M$ be a fixed point and
let $\omega$ be a smooth form on $M$ such that $\Delta\omega=\lambda\omega$. If $\int_{B_p(R)}|\omega|^2$ is of subexponential growth with respect to $R$, then $\lambda\in\sigma(k,\Delta)$. If in addition to the above,  $\int_{M}|\omega|^2 =\infty$, then $\lambda\in\sigma_{\rm ess}(k,\Delta)$.
\end{corl}

\begin{proof} We consider smooth cut-off functions $\rho_n$ such that
\[
\rho_n=\left\{
\begin{array}{ll}
1& n<d(x,p)<n+1\\
0 & d(x,p)<n-1 \text{ or } d(x,p)>n+2
\end{array}
\right.,
\]
and $|\nabla \rho_n|\leq 2$. Let $\omega_n=\rho_n\omega$. Then by a similar argument as in~\eqref{347}, we have
\[
|((\Delta-\lambda)(\omega_n),(\Delta+1)^{-\gamma}(\omega_n))|\leq C
\|\omega\|_{B_{p}(n+1)\backslash B_{p}(n)}\cdot \|\omega_n\|_{L^2(M)}.
\]
Since $\int_{B_p(R)}|\omega|^2$ is of subexponential growth,
then for any $\eps>0$,
there are infinitely many $n_j$'s such that
\[
\|\omega\|_{B_{p}(n_j+1)\backslash B_{p}{(n_j)}}\leq\eps\|\omega_{n_j}\|_{L^2(M)}.
\] Again by Corollary~\ref{cor21}, this completes the proof.
\end{proof}

\begin{remark}	
An important aspect of both Proposition~\ref{84} and Corollary~\ref{85}, is that they require no curvature assumptions on the manifold. This is made possible by our generalized Weyl Criterion. The classical Weyl criterion on the other hand, would have required at least a lower bound on the Ricci curvature as well as Cheeger-Colding theory in order to reach the same result.
\end{remark}

Proposition~\ref{84} allows us to make the following observation with regards to Conjecture~\ref{16}. Consider  a complete Riemannian manifold $M$, of dimension $n$, and assume that   $\sigma_{\rm ess}(k,\Delta)\neq\emptyset$. By Proposition~\ref{84}, a generalized eigenvalue of the Laplacian on $B_{p_i}(R_i)$, in the sense that $\Delta\omega_i=\lambda\omega_i$, will belong to the essential spectrum whenever $
\int_{B_{p_i}(R_i')}|\omega_i|^2$
is at least a certain proportion of the total
$\int_{B_{p_i}(R_i)}|\omega_i|^2$.
In other words,  whenever $\omega_i$ does not concentrate on the annulus $B_{p_i}(R_i)\backslash B_{p_i}(R_i')$, then Conjecture ~\ref{16} should be true, because any number above the bottom of the essential spectrum would belong to the spectrum of the Laplacian.
If assume the opposite  i.e. $\omega_i$ concentrates on the annulus)  then $\omega_i$ would satisfy the conditions
\begin{align*}
&\Delta\omega_i=\lambda\omega_i,\qquad \int_{B_{p_i}(R_i)\backslash B_{p_i}(R_i')}\|\omega_i\|^2=1\\
&\omega_i|_{\pa B_{p_i}(R_i)}=d^*\omega_i|_{\pa B_{p_i}(R_i)}=0\\
&\|\omega_i\|_{L^\infty(\pa B_{p_i}(R_i'))}+\|\nabla\omega_i\|_{L^\infty(\pa B_{p_i}(R_i'))}\leq\eps
\end{align*}
for some $\eps$ sufficiently small. The last
inequality follows from the fact that $
\int_{B_{p_i}(R_i')}|\omega_i|^2$ is sufficiently small.

If $\eps=0$, then by  the maximum principle, $\omega_i\equiv 0$, which is a contradiction. So the proof of   Corollary \ref{16} is related to an effective  version of the maximum principle, which we formulate it in terms of the following conjecture (for functions only) precisely.

\begin{conjecture}\label{87}
Let $M$ be a compact Riemannian manifold with smooth boundary. Assume that the Ricci curvature of $M$ is nonnegative; the second fundamental form of the boundary is bounded; and the diameter of $M$ is 1. Let $f$ be a function on $M$ such that
\[
\Delta f=\lambda f.
\]
Moreover, there is an $\eps>0$ sufficiently small such that
\[
\qquad  \|f\|_{L^\infty(\pa M)} \leq \eps, \qquad \text{and} \qquad \|\nabla f\|_{L^\infty( D)}\leq \eps
\]
on $D\subset M$ such that ${\rm meas}(D)$ is large. Then,
\[
\|f\|_{L^\infty(M)}\leq \Psi(\eps\mid n, \lambda)/\Psi_1({\rm meas}(D) \mid n,\lambda),
\]
where $\Psi$, $\Psi_1$ are functions as in \eqref{psi}.
\end{conjecture}

We believe that the conjecture can be proved using the Cheeger-Colding Theory, and
 anticipate that this approach would give a path towards the proof of Conjecture~\ref{16}. We  plan to further explore it in future work.\\

 \appendix

\section{A generalized Weyl criterion}

Let~$H$ be a  densely defined,  self-adjoint and nonnegative operator on a Hilbert space~$\Hilbert$.
The norm and inner product on~$\Hilbert$ are respectively
denoted by~$\|\cdot\|$ and $(\cdot,\cdot)$. Let $\Dom(H)$ denote the domain of $H$. Here we  prove a qualitatively and quantitatively stronger criterion  to locate the spectrum of $H$ than that  in~\cite{char-lu-1}.  In the paper~\cite{char-lu-1}, we called such a criterion a \emph{generalized Weyl Criterion}, in contrast to the \emph{classical Weyl criterion}, and we  discussed the limitations of the latter.  For a different version of the generalized Weyl criterion see~\cite{KrLu}, and for a  further application  see \cite{ChLu6}.

\begin{thm}[The Classical Weyl Criterion]
Let $H$ be defined as above. We fix  $\lambda\geq 0$, and $\delta>0$. Then
\[
{\rm dist}(\lambda,\sigma(H))<\delta
\]
if and only if there exists a sequence $\{\psi_j\}_{j \in \Nat} \subset  \Dom(H)$ with $\|\psi_j\|=1 \ \ \forall \, j\in\Nat$, such that
\[
\|(H-\lambda)\psi_j\|\leq\delta.
\]
Moreover,
\[
{\rm dist}(\lambda,\sigma_{\rm ess}(H))<\delta
\]
if and only if, in addition to the above inequality, we have $\psi_j\to 0$ weakly as $j\to\infty$.
\end{thm}

We call $\psi_j$ an \emph{approximate} eigenfunction of $\lambda$. In the case that $H$ is the Hodge Laplacian, we call $\psi_j$ an \emph{approximate} eigenform.

Let $f, g$ be two  bounded positive continuous functions on $[0,\infty)$, with $g$ satisfying the following additional property: for any $\lambda\geq 0$,  there exists a positive  constant $c$ such that $g(t)(t-\lambda)\geq c>0$ on the interval $[\lambda+1,\infty)$. We define the following three constants: for a fixed $\lambda\geq 0$ we set
\begin{align}\label{constants}
\begin{split}
 c_0=\max(\sup f(t), \sup g(t)) \, ; \ \ \ \ \qquad  \qquad & \ c_1(\lambda)=\inf_{t\in[0,\lambda]} f(t)\, ;  \\
c_2(\lambda)=\min\left(\inf_{t\in[\lambda,\lambda+1]} g(t), \inf_{t\in[\lambda+1,\infty)} g(t) (t-\lambda)\right); \ \ \ 
 &c_3(\lambda)=\lambda\,\sup_{t\in[0,\lambda]} g(t)\, .
\end{split}
\end{align}

\begin{thm}[The Generalized Weyl Criterion]\label{Thm.Weyl.bis-4}
Let $H$ be defined as above. We fix  $\lambda\geq 0$, and $0<\delta<c_0$. If
\[
{\rm dist}(\lambda,\sigma(H))<\delta/c_0,
\]
then there exists a sequence $\{\psi_j\}_{j \in \Nat} \subset  \Dom(H)$ with $\|\psi_j\|=1 \ \ \forall \, j\in\Nat$, such that
\begin{enumerate}
\item
$
 |(f(H) (H-\lambda)\psi_j, (H-\lambda)\psi_j)|\leq \delta,  \quad {and}
$
\item
$
|(g(H)\psi_j, (H-\lambda)\psi_j)|\leq \delta.
$

\end{enumerate}

Whenever ${\rm dist}(\lambda,\sigma_\mathrm{ess}(H))<\delta/c_0,$  then in addition to the above properties, we have
\begin{enumerate}
\setcounter{enumi}{2}
\item
$
  \psi_j \to 0, \text{ weakly as } j\to\infty
$
\text{ in } $\mathcal H$.
\end{enumerate}

On the other hand, if properties (1),(2) are satisfied for a sequence of $\{\psi_j\}_{j \in \Nat} \subset \Dom(H)$ with $\|\psi_j\|=1 \ \ \forall \, j\in\Nat$, and $c_1(\lambda),  c_2(\lambda)>0$ , then for some constant $c(\lambda)>0$
\[
{\rm dist}(\lambda,\sigma(H))\leq c(\lambda)\cdot\delta^{1/3}
\]
in the case $\lambda>0$. In the case $\lambda=0$ we can show  $\ {\rm dist}(\lambda,\sigma(H))\leq  \delta /c_2(\lambda).$

If the $\{\psi_j\}_{j \in \Nat}$ also satisfy  (3), then
the above upper bounds also hold  for ${\rm dist}(\lambda,\sigma_\mathrm{ess}(H))$.
\end{thm}

 In this paper we take
$f(t)=(t+\alpha)^{-2}$ and $g(t)=(t+\alpha)^{-1}$ for some positive number $\alpha$.
Observe that for any integer  $\gamma\geq 0$, we have
\begin{equation} \label{cor21e1}
\begin{split}
&((H+\alpha)^{-\gamma} (H-\lambda)\psi_j, (H-\lambda)\psi_j) \\
&=((H+\alpha)^{-\gamma+1}\psi_j, (H-\lambda)\psi_j)-(\alpha+\lambda)\;((H+\alpha)^{-\gamma}\psi_j, (H-\lambda)\psi_j).
\end{split}
\end{equation}
Using the above identity, we obtain the following useful criterion
\begin{corl}\label{cor21} A nonnegative real number  $\lambda$ belongs to the spectrum $\sigma(H)$ if, and only if, there exists a  constant $\alpha>0$ and   $\{\psi_j\}_{j \in \Nat} \subset \Dom(H)$ with $\|\psi_j\|=1 \ \ \forall \, j\in\Nat$, such that
\begin{enumerate}
\item
$
((H+\alpha)^{-\gamma}\psi_j, (H-\lambda)\psi_j)\to 0 \
$
for $\gamma=1,2$.
\end{enumerate}
Moreover, $\lambda$ belongs   $\sigma_\mathrm{ess}(H)$  if, and only if,
in addition to the above properties

\begin{enumerate}
\setcounter{enumi}{1}
\item $  \psi_j \to 0, \text{ weakly as } j\to\infty$\text{ in } $\mathcal H$.
\end{enumerate}
Furthermore, if for some $0<\delta<1$,
\[
|((H+\alpha)^{-\gamma}\psi_j, (H-\lambda)\psi_j)|\leq\delta
\]
for  $\gamma=1,2$ and  for all $j$, then there exists a constant $c(\lambda,\alpha)>0$, such that
\[
{\rm dist}(\lambda,\sigma(H))<c(\lambda,\alpha)\cdot \delta^{1/3}.
\]
\end{corl}

\begin{proof}[Proof of the Theorem~\ref{Thm.Weyl.bis-4}] The proof of the first part of the theorem is identical to that in~\cite{char-lu-1}. For the reverse statement   we use the assumptions on $H$ to write  $H=\int_0^\infty t\, dE(t)$ for the spectral measure $E$ of $H$.

When $\lambda>0$, we assume that ${\rm dist}(\lambda,\sigma(H))>\eps$  for some $\eps<\min(\lambda,1)$. Let $P=E([0,\lambda-\eps])$. $P$ is  the orthogonal  projection operator which can be used  to write $\psi_j=\psi_j^1+\psi_j^2,$ where  $\psi_j^1 =P \psi_j,$ and $\psi_j^2=\psi_j-\psi_j^1$.

By the spectral decomposition, given the assumptions on $f$ and $g$ we get
\begin{align*}
&(f(H) (H-\lambda)\psi_j, (H-\lambda)\psi_j) \geq c_1(\lambda)\eps^2\|\psi_j^1\|^2 \ \ \ \text{and} \\
&(g(H)\psi_j, (H-\lambda)\psi_j)\geq c_2(\lambda)\eps\|\psi_j^2\|^2-c_3(\lambda)\|\psi_j^1\|^2.
\end{align*}
If the criteria {\it (1)}, {\it (2)} are satisfied, then, by the two inequalities above
 \begin{align*}
\delta\geq  c_1(\lambda)\eps^2\, x\ \  \text{and} \ \
\delta\geq c_2(\lambda)\eps(1-x)-c_3(\lambda) x,
 \end{align*}
 where $x=\|\psi_j^1\|^2$. Using the first inequality to eliminate $x$ from the second, we have
 \[
 \delta\geq\frac{c_1(\lambda)\, c_2(\lambda)\,\eps^3}{c_1(\lambda)\eps^2+c_2(\lambda)\eps+c_3(\lambda)}\geq  c(\lambda)\,\eps^3
 \]
which proves the upper bound for $\eps$. If {\it (3)} is satisfied, then the estimate holds for  $\sigma_\mathrm{ess}(H)$.

In the case $\lambda=0$ we have  $c_3(\lambda)=x=0$, and we get the estimate $\eps<\delta/c_2(\lambda)$.
\end{proof}

\section{Proof of Theorem \ref{thmSpecGH}}

\begin{proof} We  consider  the two operators that make up the Laplacian $\Delta$ on $k$-forms
\[
\mathcal{L}^1=\delta d, \ \ \mathcal{L}^2=   d \delta.
\]
Each one of the above operators has a self-adjoint Friedrichs extension which is nonnegative and
$\mathfrak{Dom}(k,\Delta)= \mathfrak{Dom}(k,\mathcal{L}^1)\cap  \mathfrak{Dom}(k,\mathcal{L}^2)$. By the Hodge decomposition theorem on complete manifolds  \cites{mazz,ChLu6} for any $0\leq k \leq n$ we get
\begin{equation}\label{346}
\sigma_{\rm ess}  (k,\mathcal{L}^1)\cup \sigma_{\rm ess}   (k,\mathcal{L}^2)\setminus \{0\}\subset\sigma_{\rm ess}  (k,\Delta) \subset \sigma_{\rm ess}    (k,\mathcal{L}^1)\cup \sigma_{\rm ess}   (k,\mathcal{L}^2).
\end{equation}

Let $\Delta_i$ be the Laplacian corresponding to the metric $g_i$ on $k$-forms. Since $g_i$ is the product metric, we have $\sigma(k,\Delta_i)= \sigma_{\rm ess}(k,\Delta_i)=[\alpha(N_i,m,n,k),\infty).$
By Poincar\'e duality, we have $\sigma_{\rm ess}(k,\mathcal{L}^2) =\sigma_{\rm ess}(n -k,\mathcal{L}^1)$. Using this, we can reduce the proof of the theorem to the following version.

We define  $\mathcal{L}^1_M= \delta_M d_M$ and  $\mathcal{L}^1_i= \delta_id_i$, where $d_M,\delta_M$ (resp. $d_i, \delta_i$) are the exterior derivative operator and its adjoint on the manifold $M$ (resp. $N_i\times\R^{m}$).
By~\eqref{346}, to prove the theorem it suffices to show the following inclusion:
\[
\sigma_\mathrm{ess}(k,\mathcal{L}^1_M)\supset \bigcap_{j=1}^\infty\bigcup_{i>j}\sigma_{\rm ess}(k,\mathcal{L}^1_{i}) .
\]

We will use  the generalized Weyl criterion to prove the above containment.
Denote the $L^2$ pairing on $(M,g_M)$ by $( \cdot \, , \cdot)_M$  and on $(N_i\times\R^{m}, g_i)$   by $( \cdot \,, \cdot)_i$, and the respective norms by $\|\cdot\|_M$ and $\|\cdot\|_i$. Let $\lambda>0$ be a positive number and assume that
\[
\lambda \in  \bigcap_{j=1}^\infty\bigcup_{i>j}\sigma_{\rm ess}(k,\mathcal{L}^1_{i}).
\]
Then there is a subsequence  of positive integers $a_j$ such that $\lambda\in \sigma_{\rm ess}(k,\mathcal{L}^1_{a_j})$ for all $j$.
For simplicity we rename $a_j=j$. By Corollary \ref{cor21},  there exists a sequence of compactly supported $k$-forms $\{\psi_j\}_{j\in\mathbb{N}}$ with $\|\psi_j\|_j=1$ such that for $\gamma=1,2$
\begin{equation} \label{thm7e2}
\bigl|\,(\,(\mathcal{L}^1_{j}+1)^{-\gamma}\psi_j,(\mathcal{L}^1_{j}-\lambda) \psi_j)_{j} \,\bigr|\to 0.
\end{equation}

Let $p_j=f_j(x_j)$. Since  the $N_j$  have a uniformly bounded diameter, we may assume, without loss of generality,  that there exists a sequence of numbers $R_j\to\infty$ such that  the support of $\psi_{j}$ lies in $B_{p_j}(R_{j})\subset {N_j\times\R^{m}}$ and the support of  $\w_{j}=(f_{j})^*(\psi_{j})$ lies in $ B_{x_{j}}(R_{j})$.

Let $Q^1_M$ and $Q^1_{j}$ denote the associated quadratic forms of $\mathcal{L}^1_M$ and $\mathcal{L}^1_j$ respectively.  Then $Q^1_M (\omega)=(d_M \omega,d_M \omega)_M$ and $Q^1_{j}(\omega)=(d_j \omega,d_j \omega)_j$. For ${j}$ large enough
\begin{equation}\label{thm7e3}
\begin{split}
\bigl| Q^1_M(\omega_j)- Q^1_{j}(\psi_{j}) \bigr|  = \bigl| \|(f_{j})^*(d_{j}\psi_{j})\|_M^2 -  \|d_{j}\psi_{j}\|_{j}^2 \bigr|  \leq \Psi(j^{-1}) \|d_{j}\psi_{j}\|_{j}^2
\end{split}
\end{equation}
by assumption \eref{thm7e1} and the fact that the exterior derivative is independent of the metric.

Moreover, using a perturbation argument as in \cite{ChLu6} one can show that under assumption \eref{thm7e3} the resolvent operators also satisfy an $\eps$-approximation estimate of the type
\begin{equation*}
\bigl| \bigl(\,(\mathcal{L}_M^1+1)^{-\gamma} \omega_j, \omega_j\bigr)_M  - \bigl(\,(\mathcal{L}^1_j+1)^{-\gamma} \psi_{j}, \psi_{j}\bigr)_{j}  \bigr| \leq \, \Psi(j^{-1}) \, \|\psi_{j}\|_{j}^2
\end{equation*}
for $\gamma=1, 2$. Since,
\[
\bigl(\, (\mathcal{L}_M^1+1)^{-\gamma} \omega_j,(\mathcal{L}_M^1-\lambda)\omega_j\bigr)_M =\bigl(\,(\mathcal{L}_M^1+1)^{-\gamma+1} \omega_j, \omega_j\bigr)_M-(\lambda+1)\bigl(\,(\mathcal{L}_M^1+1)^{-\gamma} \omega_j,\omega_j\bigr)_M
\]
and a similar identity is true for $\mathcal L_j^1$,  we get
\begin{equation*}
\bigl| \bigl(\, (\mathcal{L}_M^1+1)^{-\gamma} \omega_j,(\mathcal{L}_M^1-\lambda)\omega_j\,\bigr)_M - \bigl(\,(\mathcal{L}^1_{j}+1)^{-\gamma} \psi_{j},(\mathcal{L}^1_{j}-\lambda) \psi_{j}\,\bigr)_{j} \bigr| \leq\Psi(j^{-1}) \|\omega_j\|_M^2
\end{equation*}
for $\gamma=1,2$. By Corollary~\ref{cor21}, we have $\lambda\in\sigma_\mathrm{ess}(k,\mathcal{L}^1_M)$.  Moreover, since the metrics of $B_{x_i}(R_i)$ and $N_i\times\R^m$ are very close, we have $\alpha_k={\displaystyle \liminf_{i\to\infty}}\, \lambda_o(k, B_{x_i}(R_i))$.

If $S$ is a minimal sequence, then $\alpha_k=\lambda_o^{\rm ess}$. Thus~\eqref{279} holds.  This completes the proof of the theorem.
\end{proof}



\begin{bibdiv}
\begin{biblist}

\bib{Ant}{article}{
   author={Antoci, Francesca},
   title={On the spectrum of the Laplace-Beltrami operator for $p$-forms for
   a class of warped product metrics},
   journal={Adv. Math.},
   volume={188},
   date={2004},
   number={2},
   pages={247--293},

}

\bib{BKN}{article}{
   author={Bando, Shigetoshi},
   author={Kasue, Atsushi},
   author={Nakajima, Hiraku},
   title={On a construction of coordinates at infinity on manifolds with
   fast curvature decay and maximal volume growth},
   journal={Invent. Math.},
   volume={97},
   date={1989},
   number={2},
   pages={313--349},
}

\bib{Bus}{article}{
   author={Buser, Peter},
   title={A note on the isoperimetric constant},
   journal={Ann. Sci. \'{E}cole Norm. Sup. (4)},
   volume={15},
   date={1982},
   number={2},
   pages={213--230},
   issn={0012-9593},
   review={\MR{683635}},
}

\bib{CharJFA}{article}{
   author={Charalambous, Nelia},
   title={On the $L^p$ independence of the spectrum of the Hodge
   Laplacian on non-compact manifolds},
   journal={J. Funct. Anal.},
   volume={224},
   date={2005},
   number={1},
   pages={22--48},
}

\bib{CLL}{article}{
   author={Charalambous, Nelia},
   author={Leal, Helton},
   author={Lu, Zhiqin},
   title={Spectral gaps on complete Riemannian manifolds},
   conference={
      title={Geometry of submanifolds},
   },
   book={
      series={Contemp. Math.},
      volume={756},
      publisher={Amer. Math. Soc., Providence, RI},
   },
   date={[2020] \copyright 2020},
   pages={57--67},
   review={\MR{4186938}},
   doi={10.1090/conm/756/15196},
}

\bib{char-lu-1}{article}{
   author={Charalambous, Nelia},
   author={Lu, Zhiqin},
   title={On the spectrum of the Laplacian},
   journal={Math. Ann.},
   volume={59},
   date={2014},
   number={1-2},
   pages={211--238},
}




\bib{ChLu5}{article}{
   author={Charalambous, Nelia},
   author={Lu, Zhiqin},
   title={The spectrum of the Laplacian on forms over flat manifolds},
   journal={Math. Z.},
   volume={296},
   date={2020},
   number={1-2},
   pages={1--12},
   issn={0025-5874},
   review={\MR{4140728}},
   doi={10.1007/s00209-019-02407-5},
}



\bib{ChLu6}{article}{
   author={Charalambous, Nelia},
   author={Lu, Zhiqin},
   title={The spectrum of continuously perturbed operators and the Laplacian
   on forms},
   journal={Differential Geom. Appl.},
   volume={65},
   date={2019},
   pages={227--240},
   issn={0926-2245},
}

\bib{CCoI}{article}{
   author={Cheeger, Jeff},
   author={Colding, Tobias H.},
   title={On the structure of spaces with Ricci curvature bounded below. I},
   journal={J. Differential Geom.},
   volume={46},
   date={1997},
   number={3},
   pages={406--480},
   issn={0022-040X},
   review={\MR{1484888}},
}

\bib{CCoII}{article}{
   author={Cheeger, Jeff},
   author={Colding, Tobias H.},
   title={On the structure of spaces with Ricci curvature bounded below. II},
   journal={J. Differential Geom.},
   volume={54},
   date={2000},
   number={1},
   pages={13--35},
}

\bib{CCoIII}{article}{
   author={Cheeger, Jeff},
   author={Colding, Tobias H.},
   title={On the structure of spaces with Ricci curvature bounded below.
   III},
   journal={J. Differential Geom.},
   volume={54},
   date={2000},
   number={1},
   pages={37--74},
}



\bib{CE}{book}{
   author={Cheeger, Jeff},
   author={Ebin, David G.},
   title={Comparison theorems in Riemannian geometry},
   note={Revised reprint of the 1975 original},
   publisher={AMS Chelsea Publishing, Providence, RI},
   date={2008},
   pages={x+168},
}
	


\bib{CFG}{article}{
   author={Cheeger, Jeff},
   author={Fukaya, Kenji},
   author={Gromov, Mikhael},
   title={Nilpotent structures and invariant metrics on collapsed manifolds},
   journal={J. Amer. Math. Soc.},
   volume={5},
   date={1992},
   number={2},
   pages={327--372},
}


\bib{CLY}{article}{
   author={Cheng, Siu Yuen},
   author={Li, Peter},
   author={Yau, Shing Tung},
   title={On the upper estimate of the heat kernel of a complete Riemannian
   manifold},
   journal={Amer. J. Math.},
   volume={103},
   date={1981},
   number={5},
   pages={1021--1063},
}

\bib{con}{article}{
   author={Conner, P. E.},
   title={The Neumann's problem for differential forms on Riemannian
   manifolds},
   journal={Mem. Amer. Math. Soc.},
   volume={20},
   date={1956},
   pages={56},
   issn={0065-9266},
   review={\MR{78467}},
}


\bib{dod}{article}{
   author={Dodziuk, Jozef},
   title={Eigenvalues of the Laplacian on forms},
   journal={Proc. Amer. Math. Soc.},
   volume={85},
   date={1982},
   number={3},
   pages={437--443},
}
	
		
\bib{Don81}{article}{
   author={Donnelly, Harold},
   title={On the essential spectrum of a complete Riemannian manifold},
   journal={Topology},
   volume={20},
   date={1981},
   number={1},
   pages={1--14},
}

\bib{Don}{article}{
   author={Donnelly, Harold},
   title={Spectrum of the Laplacian on asymptotically Euclidean spaces},
   journal={Michigan Math. J.},
   volume={46},
   date={1999},
   number={1},
   pages={101--111},
}


\bib{Don2}{article}{
   author={Donnelly, Harold},
   title={The differential form spectrum of hyperbolic space},
   journal={Manuscripta Math.},
   volume={33},
   date={1980/81},
   number={3-4},
   pages={365--385},
}

	



\bib{EF93}{article}{
   author={Escobar, Jos{\'e} F.},
   author={Freire, Alexandre},
   title={The differential form spectrum of manifolds of positive curvature},
   journal={Duke Math. J.},
   volume={69},
   date={1993},
   number={1},
   pages={1--41},
}



\bib{Fuk1}{article}{
   author={Fukaya, Kenji},
   title={Hausdorff convergence of Riemannian manifolds and its
   applications},
   conference={
      title={Recent topics in differential and analytic geometry},
   },
   book={
      series={Adv. Stud. Pure Math.},
      volume={18},
      publisher={Academic Press, Boston, MA},
   },
   date={1990},
   pages={143--238},
}

\bib{Fuk2}{article}{
   author={Fukaya, Kenji},
   title={Collapsing of Riemannian manifolds and eigenvalues of Laplace
   operator},
   journal={Invent. Math.},
   volume={87},
   date={1987},
   number={3},
   pages={517--547},
   issn={0020-9910},
   review={\MR{874035}},
   doi={10.1007/BF01389241},
}







\bib{Gr1}{book}{
   author={Gromov, Misha},
   title={Metric structures for Riemannian and non-Riemannian spaces},
   series={Modern Birkh\"auser Classics},
   edition={Reprint of the 2001 English edition},
   note={Based on the 1981 French original;
   With appendices by M. Katz, P. Pansu and S. Semmes;
   Translated from the French by Sean Michael Bates},
   publisher={Birkh\"auser Boston, Inc., Boston, MA},
   date={2007},
   pages={xx+585},
}



\bib{GuM}{article}{
   author={Guillarmou, Colin},
   author={Mazzeo, Rafe},
   title={Resolvent of the Laplacian on geometrically finite hyperbolic
   manifolds},
   journal={Invent. Math.},
   volume={187},
   date={2012},
   number={1},
   pages={99--144},
   issn={0020-9910},
}
	

\bib{Ho}{article}{
   author={Honda, Shouhei},
   title={Spectral convergence under bounded Ricci curvature},
   journal={J. Funct. Anal.},
   volume={273},
   date={2017},
   number={5},
   pages={1577--1662},
   issn={0022-1236},
}


	\bib{KS}{article}{
   author={Kowalski, Old\v{r}ich},
   author={Sekizawa, Masami},
   title={On the geometry of orthonormal frame bundles},
   journal={Math. Nachr.},
   volume={281},
   date={2008},
   number={12},
   pages={1799--1809},
}

\bib{KrLu}{article}{
   author={Krej\v ci\v r\'\i k, David},
   author={Lu, Zhiqin},
   title={Location of the essential spectrum in curved quantum layers},
   journal={J. Math. Phys.},
   volume={55},
   date={2014},
   number={8},
   pages={083520, 13},
   issn={0022-2488},
}


\bib{LL-2}{unpublished}{
   author={Leal, Helton},
   author={Lu, Zhiqin},
   title={Spectral Gaps on Laplacian on Differential Forms},
   year={2021},
   }


\bib{Post2}{article}{
   author={Lled{\'o}, Fernando},
   author={Post, Olaf},
   title={Existence of spectral gaps, covering manifolds and residually
   finite groups},
   journal={Rev. Math. Phys.},
   volume={20},
   date={2008},
   number={2},
   pages={199--231},
   issn={0129-055X},
}



\bib{lott}{article}{
   author={Lott, John},
   title={Collapsing and the differential form Laplacian: the case of a
   smooth limit space},
   journal={Duke Math. J.},
   volume={114},
   date={2002},
   number={2},
   pages={267--306},
}

\bib{lott1}{unpublished}{
  author={Lott, John},
  title={Collapsing and the differential form Laplacian: the case of a
   singular limit space},
  note={Preprint, https://math.berkeley.edu/~lott/sing.pdf},
  }


\bib{Lott2}{article}{
   author={Lott, John},
   title={On the spectrum of a finite-volume negatively-curved manifold},
   journal={Amer. J. Math.},
   volume={123},
   date={2001},
   number={2},
   pages={185--205},
}

\bib{LoSh}{article}{
   author={Lott, John},
   author={Shen, Zhongmin},
   title={Manifolds with quadratic curvature decay and slow volume growth},
   language={English, with English and French summaries},
   journal={Ann. Sci. \'{E}cole Norm. Sup. (4)},
   volume={33},
   date={2000},
   number={2},
   pages={275--290},
}




\bib{Lu-Zhou_2011}{article}{
   author={Lu, Zhiqin},
   author={Zhou, Detang},
   title={On the essential spectrum of complete non-compact manifolds},
   journal={J. Funct. Anal.},
   volume={260},
   date={2011},
   number={11},
   pages={3283--3298},
}

\bib{mazz}{article}{
   author={Mazzeo, Rafe},
   author={Phillips, Ralph S.},
   title={Hodge theory on hyperbolic manifolds},
   journal={Duke Math. J.},
   volume={60},
   date={1990},
   number={2},
   pages={509--559},
   issn={0012-7094},
}


\bib{Oneill}{article}{
   author={O'Neill, Barrett},
   title={The fundamental equations of a submersion},
   journal={Michigan Math. J.},
   volume={13},
   date={1966},
   pages={459--469},
}


\bib{Pet}{book}{
   author={Petersen, Peter},
   title={Riemannian geometry},
   series={Graduate Texts in Mathematics},
   volume={171},
   publisher={Springer-Verlag, New York},
   date={1998},
   pages={xvi+432},
   isbn={0-387-98212-4},
   review={\MR{1480173}},
   doi={10.1007/978-1-4757-6434-5},
}


\bib{Post1}{article}{
   author={Post, Olaf},
   title={Periodic manifolds with spectral gaps},
   journal={J. Differential Equations},
   volume={187},
   date={2003},
   number={1},
   pages={23--45},

}



\bib{ruh}{article}{
   author={Ruh, Ernst A.},
   title={Almost flat manifolds},
   journal={J. Differential Geom.},
   volume={17},
   date={1982},
   number={1},
   pages={1--14},
   issn={0022-040X},
   review={\MR{658470}},
}
	

\bib{SchoTr}{article}{
   author={Schoen, Richard},
   author={Tran, Hung},
   title={Complete manifolds with bounded curvature and spectral gaps},
   journal={J. Differential Equations},
   volume={261},
   date={2016},
   number={4},
   pages={2584--2606},
}

\bib{sturm}{article}{
   author={Sturm, Karl-Theodor},
   title={On the $L^p$-spectrum of uniformly elliptic operators on
   Riemannian manifolds},
   journal={J. Funct. Anal.},
   volume={118},
   date={1993},
   number={2},
   pages={442--453},
   issn={0022-1236},
}
	
 	
	\bib{wang}{article}{
   author={Wang, Jiaping},
   title={The spectrum of the Laplacian on a manifold of nonnegative Ricci
   curvature},
   journal={Math. Res. Lett.},
   volume={4},
   date={1997},
   number={4},
   pages={473--479},
   issn={1073-2780},
   review={\MR{1470419}},
   doi={10.4310/MRL.1997.v4.n4.a4},
}
		
 		





\end{biblist}
\end{bibdiv}
\end{document}